\documentclass[11pt]{article}

\usepackage[hidelinks]{hyperref}
\hypersetup{
  colorlinks   = true, 
  urlcolor     = blue, 
  linkcolor    = blue, 
  citecolor   = red 
}
\usepackage{amsmath,amsthm,amssymb}
\usepackage{graphicx}
\usepackage{bbm}

\usepackage{xcolor}
\usepackage[margin=1in]{geometry} 
\usepackage{enumerate,mathtools,mathrsfs,bm,graphicx}
\usepackage[numbers, super]{natbib}
\usepackage[hidelinks]{hyperref}
\newcommand{\N}{\mathbb{N}}									
\newcommand{\Z}{\mathbb{Z}}

\newcommand{\R}{\mathbb{R}}

\newcommand{\E}{\mathbb E}
\newcommand{\PP}{\mathcal P}
\newcommand{\GG}{\mathcal G}
\newcommand{\Gt}{\widetilde{\mathcal G}}
\newcommand{\BB}{\mathcal B}
\newcommand{\OO}{\mathcal O}
\newcommand{\T}{\mathbb T}
\DeclareMathOperator{\curl}{curl}
\DeclareMathOperator{\divh}{div_h}
\newcommand{\vertiii}[1]{{\left\vert\kern-0.25ex\left\vert\kern-0.25ex\left\vert #1 
    \right\vert\kern-0.25ex\right\vert\kern-0.25ex\right\vert}}
    
\newcommand{\inner}[2]{\left\langle #1, #2 \right\rangle}
\newcommand{\norm}[1]{\left\Vert #1 \right\Vert}

\newcommand{\Gi}{\mathcal{G}_{i}}

\newtheorem{theorem}{Theorem}[section]

\newtheorem{lemma}[theorem]{Lemma}
\newtheorem{proposition}[theorem]{Proposition}

\newtheorem{definition}[theorem]{Definition}

\begin{document}
	\title{Ekman Boundary Layers Under Transport Noise}
	\author{Daniel Goodair\footnote{daniel.goodair@epfl.ch} \qquad Szymon Sobczak\footnote{szymon.sobczak@epfl.ch}}
	\date{\today} 
	\maketitle
\setcitestyle{numbers}	
\thispagestyle{empty}
\begin{abstract}
We consider a 3D rotating Navier-Stokes equation with transport-stretching noise, posed between two horizontal plates, and study the joint limit of vanishing vertical viscosity and rapid rotation. Our driving noise depends only on the horizontal coordinates and its vertical component vanishes with the viscosity. Provided that the initial data is purely horizontal, we construct martingale weak solutions which converge in $L^2_{\omega}L^\infty_tL^2_x$ to the layered strong solution of a 2D stochastic Navier-Stokes equation with damping. The damping coefficient is dependent on the limit of the ratio between vertical viscosity and inverse rotation rate.

\end{abstract}
	
\tableofcontents
\thispagestyle{empty}
\newpage

\setcounter{page}{1}

\section{Introduction} \label{section introduction}

This work concerns the joint limit of vanishing vertical viscosity and rapid rotation in a 3D stochastic Navier-Stokes equation. We consider a divergence-free solution $u$ of the equation 
\begin{equation}  \label{Strat equation u}
    du_t = - B(u_t, u_t) dt - A_h u_t\, dt - \nu A_{z} u_t\, dt  - \mathcal{P} \frac{e_3 \wedge u_t}{\varepsilon}dt- \mathcal{P} \mathcal{G}^h u_t \circ d\mathcal{W}_t -  \nu^{\alpha} \mathcal{P} \mathcal{G}^z u_t \circ d\mathcal{W}_t,
\end{equation}
which represents the velocity of a fluid, posed on a domain $\mathcal{O} = \T^2 \times (0,1)$ and supplemented with the no-slip boundary condition $u = 0$ on $\partial \mathcal{O}$. The projector $\mathcal{P}$ is not exactly the classical Leray Projector, but to accommodate the natural conditions on $\mathcal{O}$, it is the orthogonal projection onto divergence-free vector fields with zero normal component on the physical boundary $\partial\mathcal{O}$ and null horizontal mean. Here and throughout, for $(x,y,z) \in \mathcal{O}$, we refer to the $x$ and $y$ directions as horizontal and the $z$ direction as vertical. Where numerical designation is convenient, the coordinates are ordered $(x_1, x_2, x_3)$. In (\ref{Strat equation u}),  $B(u_t,u_t) = \mathcal{P}\left(\left(u_t \cdot \nabla\right)u_t\right)$ is the nonlinear convective term, $A_h = -\mathcal{P}\sum_{j=1}^2 \partial_j^2$ is the horizontal Stokes Operator and $A_z = -\mathcal{P}\partial_3^2$ is the vertical Stokes Operator with turbulent vertical viscosity $\nu$. Using the superscript $j$ to denote the $j^{\textnormal{th}}$ component mapping of the vector field, the rotational term modelling Coriolis force is explicitly $e_3 \wedge u_t = (-u_t^2, u_t^1, 0)$ with frequency modulated by the constant $\frac{1}{\varepsilon}$.\\ 

In the Stratonovich stochastic integrals, $\mathcal{W}$ denotes a Cylindrical Brownian Motion acted upon by operators $\mathcal{G}^h$, $\mathcal{G}^z$ in the sense that
$$\mathcal{G}^h u_t \circ d\mathcal{W}_t = \sum_{i=1}^\infty \mathcal{G}^h_i u_t \circ dW^i_t, \qquad \mathcal{G}^z u_t \circ d\mathcal{W}_t = \sum_{i=1}^\infty \mathcal{G}^z_i u_t \circ dW^i_t $$
where $(W^i)$ is a collection of independent standard Brownian Motions comprising $\mathcal{W}$, along with pre-assigned spatial correlation functions $(\xi_i)$ with respect to which $(\mathcal{G}^h_i)$, $(\mathcal{G}^z_i)$ are defined. Each $\xi_i: \T^2 \rightarrow \R^3$, referred to as a 2D-3C vector field due to its two dimensional dependence with three components. These vector fields are smooth and divergence-free, and although $\xi_i^3$ is independent of $z$ it is not assumed trivial; in particular, the typical compatibility condition $\xi_i \cdot \underline{n} = 0$ on $\partial \mathcal{O}$ is lost, where $\underline{n}$ is the outward unit normal vector. The operators $\mathcal{G}^h_i$, $\mathcal{G}^z_i$ are defined by
\begin{equation} \label{lead to}\mathcal{G}_i^hu_t = \sum_{j=1}^2\left(\xi_i^j\partial_ju_t + u_t^j\nabla \xi_i^j\right), \qquad \mathcal{G}_i^zu_t = \xi_i^3\partial_3u_t + u_t^3\nabla \xi_i^3. \end{equation}
The power $\alpha$ in the viscous scaling of the noise satisfies $\frac{1}{2} < \alpha$, although we accommodate the case $\alpha = \frac{1}{2}$ if the ratio $\frac{\nu}{\varepsilon}$ approaches zero. We consider the convergence of solutions of (\ref{Strat equation u}) to the two dimensional equation
\begin{equation} \label{strat w}
    dw_t = - B_h(w_t, w_t) dt - A_h w_t\, dt  - \sqrt{2\beta}w_t \, dt  - \mathcal{P}_h \mathcal{G}^h w_t \circ d\mathcal{W}_t
\end{equation}
where $w:\T^2 \rightarrow \R^2$ is divergence-free and of zero-mean, $\mathcal{P}_h$ is the orthogonal projection in this 2D space onto divergence-free and zero-mean vector fields, and $B_h$ is the corresponding 2D nonlinear convective term $B_h(w_t,w_t) = \mathcal{P}_h\left(\left(w_t \cdot \nabla\right)w_t\right)$. An additional damping term is present, with intensity $0 \leq \beta < \infty$. We identify $w$ with a 3D vector field by its trivial extension. Our main result is the construction of martingale weak solutions of (\ref{Strat equation u}) which converge, as $\nu, \varepsilon \longrightarrow 0$ and $\frac{\nu}{\varepsilon} \longrightarrow \beta$, to the strong solution of (\ref{strat w}) in $L^2_{\omega}L^\infty_tL^2_x$. The precise statement can be found in Theorem \ref{main result}.

\subsection{Deterministic Theory}

In the absence of noise, equation (\ref{Strat equation u}) is a classical model in geophysical fluid mechanics often referred to as the Navier-Stokes-Coriolis equation. The model has two distinctions from the classical Navier-Stokes equation, namely the anisotropic viscosity and presence of rotation. These features work in tandem. Here, we speak of the viscosity not as the molecular kinematic viscosity but rather a turbulent viscosity, measured for example by the speed of diffusion of tracers. The rotational term models the Coriolis force, which becomes significant in the large scale problems of geophysical fluid mechanics. The Coriolis force penalises vertical diffusion, so that the vertical viscosity is much smaller than the horizontal viscosity. Therefore, the horizontal viscosity is kept constant compared to the smaller vertical viscosity $\nu$. In physical situations, $\nu$ is of the same order as $\varepsilon$ and the arising boundary layer scales like $\sqrt{\nu\varepsilon}$; it is referred to as the Ekman boundary layer, in this case of rotating fluids. Due to the same order smallness of the physical parameters $\nu$ and $\varepsilon$, the limit $\nu \rightarrow 0$, $\varepsilon \rightarrow 0$ with $\frac{\nu}{\varepsilon}$ convergent to some non-trivial $\beta$ is thus the natural asymptotic regime for deriving reduced models for applications. For a more complete discussion on the motivation of the problem, we refer the reader to [\cite{chemin2006mathematical}, \cite{greenspan1969theory}, \cite{pedlosky2013geophysical}].\\

Mathematically, the problem is particularly interesting due to the open question of the inviscid limit. Whether or not weak solutions of the Navier-Stokes equations with no-slip boundary conditions converge, as the viscosity is taken to zero, to the strong solution of the Euler equation remains one of the outstanding problems of mathematical fluid mechanics. Positive results have been limited to very specific cases regarding analyticity of initial data or structure of the flow [\cite{lopes2008vanishing1}, \cite{lopes2008vanishing}, \cite{sammartino998zero}, \cite{sammartino1998zero}], whilst conditional results such as [\cite{kato1984remarks}, \cite{kelliher2007kato}, \cite{wang2001kato}] characterise the convergence by energy dissipation in a boundary layer which is not known to hold in general. Therefore, variants of the problem which have a solution hold particular interest.\\

Grenier and Masmoudi in [\cite{grenier1997ekman}] first proved the deterministic version of our result. There is a strong intuition that as the rotation in the horizontal plane dominates the dynamics, the solution will converge towards something two dimensional. For a fixed viscosity and on the full space this is exactly what occurs, as shown in [\cite{chemin2002anisotropy}]. The mechanism of proof is dispersion; this is in contrast to our domain $\mathcal{O}$, where the fluid is constrained by the horizontal plates, and the solution simply approaches zero as also shown in [\cite{grenier1997ekman}]. This phenomenon is predicted by the Taylor-Proudman theorem, which asserts that a fluid under rapid rotation will move in `Taylor-Proudman Columns', such that the velocity is independent of $z$. Due to the no-slip boundary condition, $u(x,y,0) = u(x,y,1) = 0$, so $z$ independence of $u$ implies that it must be trivial. To recover a non-trivial limit, $\nu$ must also tend to zero with $\varepsilon$ such that $\frac{\nu}{\varepsilon}$ converges to a finite $\beta$. The singular behaviour of the fluid near the boundary for small viscosity disrupts the Taylor-Proudman Columns, and the Ekman boundary layers describe the balance between these competing phenomena. This balance produces `Ekman pumping and suction', leading to a `spin-down' effect which is modelled by the damping term in the limit equation. We again refer to [\cite{chemin2006mathematical}, \cite{greenspan1969theory}, \cite{pedlosky2013geophysical}] for a more detailed description of these phenomena.\\

The analysis of [\cite{grenier1997ekman}] relies on the construction of a boundary corrector $\mathcal{B}$ such that $w + \mathcal{B}$ is zero on the boundary, whilst $\mathcal{B}$ has $L^2_x$ norm vanishing with $\nu$ and $\varepsilon$. An integration by parts is now facilitated in energy estimates on $u - w - \mathcal{B}$, and the result is achieved through a careful analysis of the many terms involved. The damping term explicitly appears out of the construction of $\mathcal{B}$. For the limit equation to be 2D then the initial data is taken to be purely horizontal, however Masmoudi obtained a further convergence result in the case of three dimensional initial data in [\cite{masmoudi2000ekman}]. The limit equation is much more involved, and we only consider the case of horizontal initial data in this paper. Let us just mention that on $\R^3$, for fixed $\nu$ and vanishing $\varepsilon$, with three dimensional initial data then solutions converge to a 2D-3C Navier-Stokes equation as predicted by the Taylor-Proudman theorem (see [\cite{chemin2002anisotropy}]).

\subsection{Structure of the Noise} \label{subs structure}

At its core, the transport-stretching noise appearing in our equations follows the principle of \textit{Stochastic Advection by Lie Transport} introduced by Holm in [\cite{holm2015variational}]. Typically, this yields a Stratonovich integral in the velocity equation of fluid flow, given by 
$$\sum_{i=1}^\infty \mathcal{G}_i u_t \circ dW^i_t, \qquad  \mathcal{G}_iu_t = \sum_{j=1}^3\left(\xi_i^j\partial_ju_t + u_t^j\nabla \xi_i^j\right).$$
We refrain from an attempt to review the now substantial literature motivating transport noise in fluid dynamics, owing to developments in turbulence modelling, geometric mechanics, model reduction and regularisation by noise; one can see [\cite{chapron2023stochastic}, \cite{flandoli2023stochastic}] for a survey of the topic.\\

Whilst the Stochastic Advection by Lie Transport noise gives the model its foundation, alterations must be made to be suitable for our problem. As motivated in the previous subsection, we are investigating the convergence of solutions towards a two dimensional flow. To maintain this phenomenon, the spatial correlation functions appearing in the limit equation should be two dimensional. This motivates our choice for each $\xi_i$ to be independent of $z$; whilst we could consider approximations to $\xi_i$ which depend on $z$, we feel that the technical complications obfuscate the main ideas without adding any insight. A far more significant consideration is the value of $\xi_i^3$. Whilst this must be scaled to approach zero for consistency of the 2D limit equation, the way in which this is done has meaningful consequences on the dynamics and Ekman boundary layer.\\

To illustrate the idea, let us consider the It\^{o}-Stratonovich corrector appearing out of this noise; following the rigorous results of [\cite{goodair2025stratonovich}], then at least formally, the Stratonovich integral has the expression
$$-\mathcal{G}_i u_t \circ dW^i_t = -\mathcal{G}_i u_t  dW^i_t + \frac{1}{2}\mathcal{G}_i^2u_t dt.$$
At this level we are ignoring the influence of the projector or pressure and simply argue heuristically. In the It\^{o}-Stratonovich corrector, one can isolate the top order term in the vertical direction as
$$(\xi_i^3)^2\partial_3^2u_tdt.$$
Of course this is of the same order as the vertical part of the Laplacian, whose vanishing limit causes the formation of a boundary layer and damping in the limit equation. If we were to scale $\xi_i^3$ by $\nu^{\frac{1}{2}}$, the top order term would then be
$$\nu(\xi_i^3)^2\partial_3^2u_tdt$$
which matches the rate of the vertical Laplacian, therefore influencing the Ekman boundary layer and limit equation in a non-negligible way. One might anticipate that the interaction between this term and the Coriolis force produces additional damping in the limit equation, totalling
$$\sqrt{2\beta\left(1 + \frac{1}{2}\sum_{i=1}^\infty (\xi_i^3)^2\right)}w_tdt.$$
However, even an ansatz of the martingale's contribution to the limit equation is difficult. One would need a suitable theory of boundary layer expansions under transport noise, where at least to us and for the time being, it seems unclear how to match the martingale term. This direction remains of interest for the future.\\

In the present paper, we only treat this scaling in the case where $\beta = 0$. There, the vanishing viscosity overpowers the rapid rate of rotation and no damping is observed. In some sense this is the most challenging case, as the vanishing viscosity limit at the boundary could produce blow-up, which is balanced out by the regularising effect of rapid rotation; dominance of the vanishing viscosity thus appears as the scariest case. However, it kills the precise dynamics contributing to damping which are difficult to understand in the critical scaling stochastic case. Instead, when $0 < \beta < \infty$, we scale $\xi_i^3$ by $\nu^{\alpha}$ for some $\frac{1}{2} < \alpha$. This sends the corresponding noise term to zero \textit{just} quickly enough to leave no footprint in the limit equation.\\

Therefore, our noise is introduced into the 3D equation following Stochastic Advection by Lie Transport but for the scaled, effective spatial correlation functions 
\begin{equation} \label{effective spatial correlation} \tilde{\xi}_i  = \left(\xi_i^1, \xi_i^2, \nu^{\alpha}\xi_i^3\right)\end{equation}
which can alternatively be expressed as
$$\xi_i^h + \nu^{\alpha}\xi_i^z$$
where $\xi_i^h = \left(\xi_i^1, \xi_i^2, 0\right)$ and $\xi_i^z = \left(0, 0, \xi_i^3\right)$. Generalising the noise operator for vector fields $\phi$, $f$ by
\begin{equation} \label{definition of G} \mathcal{G}_{\phi}f \coloneqq \sum_{j=1}^3\left(\phi^j\partial_jf + f^j\nabla \phi^j\right),\end{equation}
then our desired stochastic term is $\sum_{i=1}^\infty \mathcal{G}_{\tilde{\xi_i}}u_t \circ dW^i_t$ where
\begin{align*}
  \mathcal{G}_{\tilde{\xi_i}}u_t &= \sum_{j=1}^3\left(\tilde{\xi}_i^j\partial_ju_t + u_t^j\nabla \tilde{\xi}_i^j\right)
  \\&= \sum_{j=1}^2\left(\xi_i^j\partial_ju_t + u_t^j\nabla \xi_i^j\right) + \nu^{\alpha}\left(\xi_i^3\partial_3u_t + u_t^3\nabla \xi_i^3\right)\\
  &= \mathcal{G}_i^hu_t + \nu^{\alpha}\mathcal{G}_i^zu_t.
\end{align*}
Thus we arrive at exactly the stochastic term of (\ref{Strat equation u}). We remark that $\mathcal{G}_i^h = \mathcal{G}_{\xi_i^h}$ and $\mathcal{G}_i^z = \mathcal{G}_{\xi_i^z}$, hence the horizontal and vertical superscripting on $\mathcal{G}_i^h$ and $\mathcal{G}_i^z$ refers to transport and stretching along the horizontal and vertical components, respectively, of the spatial correlation functions. Indeed, this decomposition of $\xi_i$ precisely isolates the directional derivatives of $u$; scaling of the vertical derivative is not an artificial choice for the analysis, but rather a consequence of scaling the underlying spatial correlations. We remark that a similar splitting was used to study an anisotropic inviscid limit in [\cite{goodair2026anisotropic}], where $\xi_i^h$ was also scaled to zero but at a different rate. Another directional decomposition of transport-stretching noise appeared in [\cite{flandoli2024boussinesq}], where the Taylor-Proudman model considered was itself 2D-3C. In this work the directions are not differently scaled, and the authors explore the Boussinesq Hypothesis through the limiting regime initiated in [\cite{galeati2020convergence}].

\subsection{Aspects of the Proof} \label{subs aspects}

To motivate a discussion on the main elements of the proof, let us first mention some of the most related results from the literature. Whilst there have been several works on stochastic Ekman layers from the numerical perspective, see for example [\cite{farrell1993stochastic}, \cite{kazemi2016dynamic}, \cite{kim2026statistical}, \cite{klein2022exploring}, \cite{klein2023capturing}, \cite{li2024stochastic}, \cite{li2025generalized}], from the analysis perspective the literature is far less developed. The work [\cite{hieber2013stochastic}] proves stability of the Ekman spiral, that is the explicit stationary solution of the deterministic Navier-Stokes-Coriolis equation, under stochastic perturbations. Closest to our result is [\cite{wang2024zero}], where the authors consider the same problem as us but for an additive noise acting only in the horizontal directions. Whilst we describe our proof in more detail below, the main behaviour that we must control owes to the vertical derivative from transport noise at the boundary; the balance between this singular behaviour and the regularisation of rotation is new to this work. An inviscid limit problem on the same domain and with transport noise was recently studied by the first author in [\cite{goodair2026anisotropic}], without rotation and where the horizontal viscosity is also sent to zero. In that work the limit was a strong solution of the deterministic Euler equation, and differences of the approach will be discussed below. For further reading on the inviscid limit of stochastic Navier-Stokes equations, we point to  [\cite{goodair2025zero}, \cite{luongo2024inviscid}, \cite{wang2024kato}] for conditional results, [\cite{cipriano2015inviscid}, \cite{goodair2025navier}] for the case of Navier boundary conditions and [\cite{bessaih2013inviscid}, \cite{brzezniak2026inviscid}, \cite{glatt2015inviscid}] for ergodic results in the absence of a boundary.\\

The first step in treating (\ref{Strat equation u}) is to convert the Stratonovich equation to It\^{o} form, where the analysis is much more favourable. For a more compact expression we will use the notation
$$\tilde{\mathcal{G}}_i \coloneqq \mathcal{G}_{\tilde{\xi}_i}$$
following (\ref{effective spatial correlation}) and (\ref{definition of G}), so that (\ref{Strat equation u}) can initially be rewritten as
\begin{equation}  \nonumber
    du_t = - B(u_t, u_t) dt - A_h u_t\, dt - \nu A_{z} u_t\, dt  - \mathcal{P} \frac{e_3 \wedge u_t}{\varepsilon}dt-\mathcal{P} \tilde{\mathcal{G}} u_t \circ d\mathcal{W}_t
\end{equation}
which corresponds, as indicated in [\cite{goodair2025stratonovich}], to the It\^{o} form
\begin{align}
    du_t = - B(u_t, u_t) dt - A_h u_t\, dt - \nu A_{z} u_t\, dt  - \mathcal{P} \frac{e_3 \wedge u_t}{\varepsilon}dt -  \mathcal{P} \tilde{\mathcal{G}} u_t \, d\mathcal{W}_t + \frac{1}{2}\sum_{i=1}^\infty \mathcal{P}\tilde{\mathcal{G}}_i\mathcal{P}\tilde{\mathcal{G}}_iu_t dt. \label{some main equation 2 Ito}
\end{align}
Whilst we will often refer to equation (\ref{some main equation 2 Ito}), it should be appreciated that the It\^{o}-Stratonovich corrector contains cross-derivatives and the equation has the full, explicit expression
\begin{align}
    \nonumber du_t = &- B(u_t, u_t) dt - A_h u_t\, dt - \nu A_{z} u_t\, dt  - \mathcal{P} \frac{e_3 \wedge u_t}{\varepsilon}dt -  \mathcal{P} \mathcal{G}^h u_t \, d\mathcal{W}_t -  \nu^{\alpha} \mathcal{P} \mathcal{G}^z u_t \, d\mathcal{W}_t\\
    &+ \frac{1}{2}\sum_{i=1}^\infty \mathcal{P}\mathcal{G}^h_i\mathcal{P}\mathcal{G}^h_iu_t dt + \frac{\nu^{\alpha}}{2}\sum_{i=1}^\infty \mathcal{P}\left( \mathcal{G}^h_i\mathcal{P}\mathcal{G}^z_iu_t + \mathcal{G}^z_i\mathcal{P}\mathcal{G}^h_iu_t \right) dt + \frac{\nu^{2\alpha}}{2}\sum_{i=1}^\infty \mathcal{P}\mathcal{G}^z_i\mathcal{P}\mathcal{G}^z_iu_t dt. \label{expanded ito form}
\end{align}
Where $\mathcal{P}$ would just be the usual Leray Projector, one typically simplifies the It\^{o}-Stratonovich corrector by
$$\mathcal{P}\tilde{\mathcal{G}}_i\mathcal{P}\tilde{\mathcal{G}}_iu_t = \mathcal{P}\tilde{\mathcal{G}}_i\tilde{\mathcal{G}}_iu_t,$$
see for example [\cite{goodair20223d}] Lemma 2.7. This property relies on the Helmholtz decomposition and the fact that the operator $\mathcal{G}_{\phi}$ preserves gradients, specifically $\mathcal{G}_{\phi}(\nabla f) = \nabla \left( (\phi \cdot \nabla) f \right)$. The same property is true when we refine the projector to also zero-mean vector fields, as $\mathcal{G}_{\phi}$ maps constant vector fields to gradients. In our case, $\mathcal{P}$ is a refinement of the Leray Projector in 3D which maps onto vector fields with null horizontal mean. One can no longer appeal to such a preservation of the orthogonal complement. This is not a trivial issue, which we highlight with the heuristic of an energy computation. To this end, let us isolate the transport and stretching parts of $\mathcal{G}$:
$$\mathcal{T}_{\phi}f \coloneqq \sum_{j=1}^3\phi^j\partial_jf, \qquad \mathcal{S}_{\phi}f \coloneqq \sum_{j=1}^3f^j\nabla \phi^j$$
with $\mathcal{T}_i \coloneqq \mathcal{T}_{\xi_i}$ and $\mathcal{S}_i \coloneqq \mathcal{S}_{\xi_i}$. In the simplest case of an $L^2$ based energy estimate, we meet the term
\begin{equation}\label{ismet}\inner{\mathcal{P}\tilde{\mathcal{G}}_i\mathcal{P}\tilde{\mathcal{G}}_iu}{u} + \norm{\mathcal{P}\tilde{\mathcal{G}}_iu}^2\end{equation}
owing to the It\^{o}-Stratonovich corrector and quadratic variation of the martingale. In the case of a pure transport noise, this is treated by
\begin{align*}
\inner{\mathcal{P}\mathcal{T}_i\mathcal{P}\mathcal{T}_iu}{u} + \norm{\mathcal{P}\mathcal{T}_iu}^2 = \inner{\mathcal{T}_i\mathcal{P}\mathcal{T}_iu}{u} + \inner{\mathcal{P}\mathcal{T}_iu}{\mathcal{T}_iu} = -\inner{\mathcal{P}\mathcal{T}_iu}{\mathcal{T}_iu} + \inner{\mathcal{P}\mathcal{T}_iu}{\mathcal{T}_iu} = 0
\end{align*}
whilst in the case of transport-stretching noise with typical Leray Projector $\mathcal{P}$, the treatment begins with
\begin{align*}
\inner{\mathcal{P}\mathcal{G}_i\mathcal{P}\mathcal{G}_iu}{u} + \norm{\mathcal{P}\mathcal{G}_iu}^2 = \inner{\mathcal{G}_i\mathcal{G}_iu}{u} + \norm{\mathcal{P}\mathcal{G}_iu}^2 \leq \inner{\mathcal{G}_i\mathcal{G}_iu}{u} + \norm{\mathcal{G}_iu}^2 = \inner{\mathcal{G}_iu}{(\mathcal{S}_i + \mathcal{S}_i^*)u}
\end{align*}
and then continues with a precise analysis, using that the remaining transport term only arises through a commutator $[\mathcal{T}_i, \mathcal{S}_i]$ which is bounded on $L^2$. Therefore, one obtains a desired bound by $\norm{u}^2$. Attempting the same argument in our case without the commutativity $\mathcal{P}\mathcal{G}_i\mathcal{P} = \mathcal{P}\mathcal{G}_i$, instead of $\mathcal{T}_i\mathcal{S}_i - \mathcal{S}_i\mathcal{T}_i$ we meet $\mathcal{T}_i\mathcal{P}\mathcal{S}_i - \mathcal{S}_i\mathcal{P}\mathcal{T}_i$ which is unbounded on $L^2$. Whilst we cannot enjoy this estimate, instead we prove that $$\inner{\mathcal{P}\tilde{\mathcal{G}}_i\mathcal{P}\tilde{\mathcal{G}}_iu}{u} +  \norm{\mathcal{P}\tilde{\mathcal{G}}_iu}^2  \leq 
\norm{\xi_i}_{W^{1,\infty}}^2\left(c_{\delta}\norm{u}^2 + \delta\norm{\nabla^hu}^2 +  \delta\nu^{2\alpha}\norm{\partial_3u}^2 \right)$$
where the derivative dependencies can be absorbed into the viscous terms of the energy estimate. We emphasise that $c_{\delta}$ is independent of $\nu$, and that $\nu$ only enters the right hand side through the scaling of $\xi_i^3$. The It\^{o} form for the 2D limit equation is \begin{equation} 
    dw_t = - B_h(w_t, w_t) dt - A_h w_t\, dt  - \sqrt{2\beta}w_t \, dt  - \mathcal{P}_h \mathcal{G}^h w_t \, d\mathcal{W}_t + \frac{1}{2}\sum_{i=1}^\infty \mathcal{P}_h\mathcal{G}^h_i\mathcal{G}^h_iw_tdt \label{Ito form for w}
\end{equation}
where we can enjoy the `weak commutativity' $\mathcal{P}_h\mathcal{G}^h_i\mathcal{P} = \mathcal{P}_h\mathcal{G}^h_i$ as discussed. A similar estimate was used in [\cite{goodair2026anisotropic}], but for a completely different reason. There, although one worked on the same domain, the projection did not include the null horizontal mean constraint; it is only required here to obtain a cancellation of the Coriolis term. Thus, the weak commutativity could be exploited in [\cite{goodair2026anisotropic}]. Rather, the difficulty in [\cite{goodair2026anisotropic}] was due to the fact that the effective spatial correlation functions $\tilde{\xi}_i$ were not divergence-free. Another issue that we addressed is that $\xi_i^3$ is constant in $z$ and non-trivial, so in particular $\xi_i \cdot \underline{n} = \pm\xi_i^3 \neq 0$ at the boundary. Typically, it is the fact that this quantity is null which facilitates the integration by parts in showing antisymmetry of $\mathcal{T}_i$, as required in the above estimates. We are careful to only need this antisymmetry property when tested against zero-trace vector fields, which carries the integration by parts in place of the impermeability of $\xi_i$.\\

Care must be taken in constructing the martingale weak solutions. An expected consequence of the lack of uniqueness for weak solutions of the 3D Navier-Stokes equations is that our solutions will be probabilistically weak. A priori, the constructed probability space and Cylindrical Brownian Motion supporting a solution will depend on $\nu$ and $\varepsilon$; such a concern can be alleviated due to an idea given in [\cite{breit-hofmanova2016}], where the application of Skorokhod's Theorem is done at the level of the family of Galerkin approximations indexed by $\nu$ and $\varepsilon$. We then show convergence to the unique strong solution of (\ref{Ito form for w}) taken with respect to this newly constructed probability space and Cylindrical Brownian Motion.\\

Following the aforementioned approach of Grenier and Masmoudi in [\cite{grenier1997ekman}], we look to estimate $u-w-\mathcal{B}$ although $u$ does not have the spatial regularity required to look at an energy identity directly. Therefore, we carry out the energy estimate at the level of the Galerkin approximation. At each order $n$, the approximate solution is driven by a different Cylindrical Brownian Motion $\mathcal{W}^n$. Labelling such a solution $u^n$, this fact renders a direct computation of $u^n - w - \mathcal{B}$ exceedingly difficult due to the non-trivial correlation between $\mathcal{W}^n$ and $\mathcal{W}$. We would prefer to estimate the difference $u^n - w^n - \mathcal{B}^n$, where $w^n$ is the strong solution driven by $\mathcal{W}^n$ and $\mathcal{B}^n$ is the boundary corrector associated to $w^n$. This is indeed what we do, by baking $w$ into the application of the Skorokhod Theorem. In fact we shall use a further approximation of $w$ by solutions with a smoother initial condition. Unlike in [\cite{wang2024zero}] where the initial condition $w_0 \in H^2$, we only assume $w_0 \in H^1$ and consider an approximate sequence of initial conditions $(w^m_0) \in H^2$. We rely on this smoothness only to handle residual terms arising from the Galerkin projections. A proper justification of the energy identity by Galerkin approximation was absent in [\cite{wang2024zero}] and indeed the original work [\cite{grenier1997ekman}], though we believe it introduces non-trivial complications worthy of detailing. The requirement that $w_0 \in H^2$ from [\cite{wang2024zero}] was instead used for a high probability control in the energy estimate, which after some computation on the nonlinear term reads as
$$\mathbbm{E}\left[\norm{u - w - \mathcal{B}}_{L^\infty([0,t];L^2)}^2 \right] \leq c\mathbbm{E}\left[\int_0^t\norm{w_s}_{H^2}^2\norm{u_s-w_s-\mathcal{B}_s}_{L^2}^2ds\right] + \dots.$$
Where $w_0 \in H^2$, then $w \in C_tH^2_x$ $\mathbbm{P}-a.s.$ and in particular for any $0 < \delta$, $\sup_{s\in[0,T]}\norm{w_s}_{H^2}^2 \leq C_{\delta}$ uniformly on a set of probability at least $1-\delta$. On this high probability set, the uniform bound can be applied and the energy estimate follows a standard Gr\"{o}nwall argument. The conclusion of [\cite{wang2024zero}] is then a convergence in probability result.\\

This is not our approach, and there are a few details to unpack. With deterministic $w_0 \in H^1$, then we do at least expect the estimate
$$\mathbbm{E}\left[\sup_{t\in[0,T]}\norm{w_t}_{H^1}^2 + \int_0^T\norm{w_s}_{H^2}^2ds  \right]\leq C\norm{w_0}_{H^1}^2$$
and one has a similar high probability control on $\int_0^T\norm{w_s}_{H^2}^2ds$. By using a Stochastic Gr\"{o}nwall Inequality from [\cite{glatt2009strong}], then this control is sufficient for the energy estimate and to conclude the convergence in probability argument. Whilst we do use the Stochastic Gr\"{o}nwall Inequality, it is only a consequence of the particular structure of our noise that we obtain not just convergence in probability but convergence in $L^2_{\Omega}$. The transport-stretching noise at velocity level in 2D reduces to a purely transport noise at vorticity level, which completely cancels in $L^2$ energy estimates thus enabling a deterministic bound. This lifts to a deterministic control
$$\sup_{t\in[0,T]}\norm{w_t}_{H^1}^2 + \int_0^T\norm{w_s}_{H^2}^2ds\leq C\norm{w_0}_{H^1}^2$$
so we can carry out the Stochastic Gr\"{o}nwall Inequality without prior restriction to a high probability set.\\

To conclude this section we briefly compare our method to that of [\cite{goodair2026anisotropic}], first noting that as the limit equation was deterministic in [\cite{goodair2026anisotropic}] then none of the above concerns applied. This leads to another significant difference, which is that our boundary corrector $\mathcal{B}$ is stochastic. We cannot apply a simple bound on $\partial_t\mathcal{B}$, but must rather identify the evolution equation satisfied by $\mathcal{B}$ and treat it term by term. We succeed by understanding the boundary corrector as a first order linear operator of $w$, applying this operator term by term in the evolution equation of $w$ which of course includes the transport-stretching noise and its It\^{o}-Stratonovich corrector. These terms require a precise control.

\section{Preliminaries}

This section is dedicated to establishing the framework of the main result, along with some required estimates. Subsection \ref{subs funct anal}
sets up the relevant function spaces, whilst Subsection \ref{boundary corrector subs} addresses the boundary corrector and its properties. Subsection \ref{subs stoch prelim} fixes notation and states results from stochastic analysis. The section concludes with Subsection \ref{subs transport stretch} providing estimates on the transport-stretching noise.

\subsection{Functional Analytic Preliminaries} \label{subs funct anal}

We recall that $\mathcal{O} = \T^2 \times (0,1)$, and denote the usual Sobolev Spaces $W^{s,p}(\mathcal{O};\R^3)$, $H^{\gamma}(\mathcal{O};\R^3)$ by simply $W^{s,p}$, $H^{\gamma}$. We shall have no quarrels in using this notation for the spaces $W^{s,p}(\T^2;\R^2)$ and $H^{\gamma}(\T^2;\R^2)$, as well as identifying vector fields $f: \T^2 \rightarrow \R^2$ with their trivial extension $\tilde{f}: \mathcal{O} \rightarrow \R^3$. We use $\inner{\cdot}{\cdot}$ to represent the $L^2$ inner product and similarly for the norm, whilst also employing subscripts $L^p_h$, $L^p_z$ as shorthand for $L^p\left(\T^2;\R^3\right)$ and $L^p\left((0,1);\R^3\right)$ respectively. This shorthand will also apply for general Euclidean target spaces, which shall be clear from the context. We proceed to define several divergence-free subspaces. For this, using the superscript $j$ to denote the $j^{\textnormal{th}}$ component mapping, we fix $\nabla \cdot f$ to be the divergence $\sum_{j=1}^3 \partial_jf^j$ and $\underline{n}$ to be the outward unit normal vector at $\partial \mathcal{O}$. Firstly, let us set
$$L^2_{\sigma} \coloneqq \left\{f \in L^2(\mathcal{O};\R^3): \nabla \cdot f = 0, \quad f \cdot \underline{n} = 0 \, \, \, \textnormal{on} \, \, \partial\mathcal{O}, \quad \int_{\T^2}f(x,y,z)dxdy = 0 \, \, \, \textnormal{for} \, \, a.e. \, z \right\}$$
where the divergence-free and boundary conditions are understood in a suitable weak sense, as the $L^2$ limit of smooth divergence-free and compactly supported functions, see for example [\cite{robinson2016three}] Section 2. Furthermore, let us define
\begin{align*}
    W^{1,2}_{\sigma} &\coloneqq \left\{f \in W^{1,2}(\mathcal{O};\R^3): \nabla \cdot f = 0, \quad f = 0 \, \, \, \textnormal{on} \, \, \partial\mathcal{O}, \quad \int_{\T^2}f(x,y,z)dxdy = 0 \, \, \, \textnormal{for} \, \, a.e. \, z \right\},\\
    \bar{W}^{1,2}_{\sigma} &\coloneqq \left\{f \in W^{1,2}(\mathcal{O};\R^3): \nabla \cdot f = 0, \quad f \cdot \underline{n} = 0 \, \, \, \textnormal{on} \, \, \partial\mathcal{O}, \quad \int_{\T^2}f(x,y,z)dxdy = 0 \, \, \, \textnormal{for} \, \, a.e. \, z \right\}
\end{align*}
as the spaces of divergence-free functions with zero horizontal mean, which are zero on the boundary or tangential to the boundary respectively. These are Hilbert Spaces under the usual $H^1$ inner product. $W^{1,2}_{\sigma}$ thus incorporates both the traditional no-slip boundary condition at the physical boundaries, and the zero-mean condition in the directions of $\T^2$. We note that the geometry of the domain means that $f \cdot \underline{n} = 0$ is equivalent to $f^3 = 0$ on $\partial \mathcal{O}$.\\

In addition, we introduce spaces specific to the two dimensional limit equation. We use $\nabla^h$ for the horizontal gradient $(\partial_1, \partial_2)$. Let us define
$$L^2_{\sigma,h} \coloneqq \left\{f \in L^2(\T^2;\R^2): \nabla^h \cdot f = 0, \quad \int_{\T^2}f(x,y)dxdy = 0 \right\}, \qquad W^{m,2}_{\sigma,h} \coloneqq W^{m,2}(\T^2;\R^2) \cap L^2_{\sigma,h}.$$
Note that if $f \in L^2_{\sigma,h}$ then its trivial extension belongs to $L^2_{\sigma}$, as we recall that the boundary condition $f \cdot \underline{n} = 0$ is equivalent to $f^3 = 0$ on $\partial \mathcal{O}$, and the extension $\tilde{f}$ is such that $\tilde{f}^3$ is zero. In continuing to identify $f$ with its extension, we shall write $f \in L^2_{\sigma}$. Similarly note that $W^{1,2}_{\sigma,h} \subset \bar{W}^{1,2}_{\sigma}$. The space of smooth, divergence-free and zero-mean vector fields is dense in $L^2_{\sigma,h}$ and every $W^{m,2}_{\sigma,h}$. We shall also employ the notation
\begin{equation} \nonumber \norm{\nabla^hf}^2 = \sum_{j=1}^2 \norm{\partial_jf}^2, \qquad \norm{\nabla f}^2 = \sum_{j=1}^3 \norm{\partial_jf}^2.\end{equation} The norms are equivalent to the usual $H^1$ norms on $W^{1,2}_{\sigma,h}$ and $W^{1,2}_{\sigma}$ respectively.\\

Furthermore, we use $\mathcal{P}$ to denote the orthogonal projection in $L^2(\mathcal{O};\R^3)$ onto $L^2_{\sigma}$ and $\mathcal{P}_h$ the orthogonal projection in $L^2(\T^2;\R^2)$ onto $L^2_{\sigma,h}$. We look to justify a consistency of these projections, namely, in continuing to use $\tilde{f}$ for the trivial extension of $f: \T^2 \rightarrow \R^2$, we claim that $\widetilde{\mathcal{P}_hf} = \mathcal{P}\tilde{f}$. Indeed, we have already noted that $\widetilde{\mathcal{P}_hf} \in L^2_{\sigma}$ so to verify that it is indeed $\mathcal{P}\tilde{f}$ we need only check that $\tilde{f} - \widetilde{\mathcal{P}_hf}$ is orthogonal to $L^2_{\sigma}$. This owes to the usual orthogonal decomposition of $f$,
$$f = \mathcal{P}_hf + \bar{f} + \nabla g$$
where $\bar{f}$ is its mean. Therefore
$$\tilde{f} - \widetilde{\mathcal{P}_hf} = \tilde{\bar{f}} + \widetilde{\nabla g}.$$
For orthogonality, we take an arbitrary $\phi \in L^2_{\sigma}$. Firstly, note that
$$\inner{\tilde{\bar{f}}}{\phi} = \int_0^1\left(\bar{f}^1\int_{\T^2}\phi^1(x,y,z)dxdy + \bar{f}^2\int_{\T^2}\phi^2(x,y,z)dxdy\right)dz = 0$$
due to the null horizontal mean condition on $\phi$. For the gradient, observe that
$$\widetilde{\nabla g} = \nabla \tilde{g}$$
where $\tilde{g}$ is simply the scalar $\tilde{g}(x,y,z) = g(x,y)$. Orthogonality with the divergence-free $\phi$, also satisfying $\phi \cdot \underline{n} = 0$ at the physical boundary, is now classical. Owing to this consistency, we will identify $\mathcal{P}_h$ with $\mathcal{P}$ and denote it simply by the latter. As a result, we identify the 2D nonlinear term $B_h$ with $B$.\\

To facilitate a Galerkin approximation, we shall also construct a specific orthonormal basis of $L^2_{\sigma}$. For this, we appeal to a Fourier decomposition in horizontal directions of $L^2_{\sigma}$. Namely, by defining $\Pi_k$ for $k \in \Z^2$ on $f: \mathcal{O} \rightarrow \R^3$ by
$$(\Pi_kf)(x,y,z) = \hat{f}(k,z)e^{ik\cdot (x,y)}$$
where $\hat{f}(k,z)$ is the $k^{\textnormal{th}}-$Fourier mode of the function $f(\cdot, z)$, then we have the decomposition
$$L^2_{\sigma} = \bigoplus_{k \in \Z^2}\Pi_kL^2_{\sigma}.$$
Note that $A_h$ is just a constant multiplier on $\Pi_kL^2_{\sigma}$, and defining $A \coloneqq A_h + A_z$ with domain $H^2 \cap W^{1,2}_{\sigma}$, then classical spectral theory allows us to construct an orthonormal basis for each $\Pi_kL^2_{\sigma}$ consisting of eigenfunctions of $A$. By taking the union over all $k$ of these bases, we have an orthonormal basis of $L^2_{\sigma}$ consisting of eigenfunctions of $A$. Moreover, as $A_h$ is just a constant multiplier on $\Pi_kL^2_{\sigma}$, then this basis also consists of eigenfunctions for $A_h$. Indexing this basis over $n \in \N$, we use $\mathcal{P}_n$ to denote the orthogonal projection in $L^2(\mathcal{O};\R^3)$ onto its first $n$ elements.\\

We end this section with a simple technical lemma, that will be useful in later estimates.
\begin{lemma}\label{lem:auxiliary-est}
    Let $f\in W^{1,2}_{\sigma}$, $g\in L^2_{\sigma,h}$, $a\in L^\infty_z$. Then the following hold:
    \begin{align}
        \| f \|_{L^2_z L^4_h}^2 &\le c \|f\| \|\nabla^hf\| \label{eq:ladyzh}\\
        \left| \int_\mathcal{O} a(z)g^l(x,y)f^j f^k\right| &\le c \|a\|_{L^\infty_z}\|g\|_{L^2_h}\|f\|\|\nabla^h f\|. \label{aux-est2}
    \end{align}
\end{lemma}
\begin{proof}
Let us start with the first inequality. By $2$-d Ladyzhenskaya's inequality, for each $z$ we have
\begin{align*}
     \| f(\cdot,z) \|_{L^4_h}^2 &\le c \|f(\cdot,z)\|_{L^2_h}\|\nabla^h f(\cdot,z)\|_{L^2_h}.
\end{align*}
Now applying the Cauchy-Schwarz inequality gives the first statement \eqref{eq:ladyzh}.

For the second statement, we start by bounding $a$ uniformly and applying a horizontal H\"older with exponents $2,4,4$, followed 
by an application of \eqref{eq:ladyzh}:
\begin{align*}
    \left| \int_\mathcal{O} a(z)g(x,y)f^j f^k\right| &\le c\|a\|_{L^\infty_z}\|g\|_{L^2_h}  \int_{[0,1]} \|f(\cdot, z)\|_{L^4_h}^2 dz
    \le c \|a\|_{L^\infty_z}\|g\|_{L^2_h}\|f\|\|\nabla^h f\|.
\end{align*}
\end{proof}





\subsection{Boundary Layer Corrector} \label{boundary corrector subs}
In this section we recall the construction of the Boundary layer corrector as carried out in [\cite{masmoudi1998euler}]. The reason why the same expansion holds, is that the vertical noise is scaled by $\nu^\alpha$ with $\alpha>\frac{1}{2}$, which makes the vertical part of the martingale terms ``small" compared to the leading order deterministic boundary layer behaviour. This intuition breaks down if one were to consider the case $\alpha=\frac{1}{2}$, which requires treating this non-trivial martingale term while dealing with the boundary layer. 

Let us recall, that the construction of the boundary layer $\mathcal{B}$ proceeds by obtaining four intermediate correctors, so that
\begin{equation}
    \mathcal{B} = \mathcal{B}_1 + \mathcal{B}_2 + \mathcal{B}_3 + \mathcal{B}_4,
\end{equation}
with $\nabla \cdot \mathcal{B} =0 $ and $\mathcal{B} + w =0$ on $\partial \mathcal{O}$. We further note, that the boundary corrector above can be seen as a linear operator acting on the effective interior solution $w$, hence warranting the notation $\mathcal{B}[w]$. Thus, retaining the same $\mathcal{M}(z)$ as in [\cite{masmoudi1998euler}] we write
\begin{equation}\label{eq:corrector-def}
\BB[f](x,y,z)=\mathcal M(z)
\begin{pmatrix}f^1(x,y)\\ f^2(x,y)\\ \curl f(x,y)\end{pmatrix}
= \sum_{i=1}^4 \BB_i[f],
\end{equation}
for any horizontal, mean zero and divergence free field $f$.The matrix \(\mathcal M\) depends on
\(\varepsilon,\nu\) and consists of a
horizontal block \(\mathcal M^h\) and a scalar block \(\mathcal M^z\). As in [Section 3.1, \cite{masmoudi1998euler}] the matrix $\mathcal{M}$ enjoys the following bounds:
\begin{equation}\label{eq:matrix}
\begin{aligned}
\|\mathcal M\|_{L_z^2}
+\|d\,\partial_3\mathcal M\|_{L_z^2}
&\le c(\varepsilon\nu)^{1/4},\\
\|\mathcal M^z\|_{L_z^\infty}
+\|d^2\partial_3\mathcal M\|_{L_z^\infty}
&\le c(\varepsilon\nu)^{1/2},\\
\|\partial_3\mathcal M\|_{L_z^2}
&\le c(\varepsilon\nu)^{-1/4},\\
\|\mathcal M\|_{L_z^\infty}
+\|\partial_3\mathcal M^z\|_{L_z^\infty}
&\le c,
\end{aligned}
\end{equation}
where $d(z)=\min\{z,1-z\}$.
Let us also collect how the bounds \eqref{eq:matrix} translate to operator bounds on $\BB$:
\begin{equation}\label{eq:Bnorm}
\begin{aligned}
\|\BB[f]\|&\le c(\varepsilon\nu)^{1/4}\|f\|_{H^1},\\
\|\nabla^h\BB[f]\|&\le c(\varepsilon\nu)^{1/4}\|f\|_{H^2},\\
\|\partial_3\BB[f]\|&\le c(\varepsilon\nu)^{-1/4}\|f\|_{H^1}.
\end{aligned}
\end{equation}

Apart from these bounds, the boundary layer correctors are constructed to display cancellations corresponding to the interaction between the turbulent vertical viscosity and fast Coriolis rotation term. The specific relations can be captured in the following:
\begin{equation}\label{eq:profiles}
\begin{aligned}
\nu\partial_3^2\BB_1[f]^h
-\varepsilon^{-1}(e_3\wedge\BB_1[f])^h&=0,\\
\nu\partial_3^2\BB_2[f]^h
-\varepsilon^{-1}(e_3\wedge\BB_2[f])^h
+\sqrt{2\nu/\varepsilon}\,f^h&=0,\\
\partial_3^2\BB_3[f]^h&=0,\\
\|\BB_3[f]^h\|&\le c e^{-1/\sqrt{2\varepsilon\nu}}\|f\|,
\end{aligned}
\end{equation}
where the superscript $h$ refers to the horizontal component of the fields.
The fourth horizontal contribution is
\begin{equation}\label{eq:profile4}
\BB_4[f]^h
=\left\{a\left(e^{-z/\sqrt{2\varepsilon\nu}}
+e^{-(1-z)/\sqrt{2\varepsilon\nu}}\right)+b\right\}(f^2,-f^1),
\quad
|a|\le c\sqrt{\varepsilon\nu},\quad |b|\le c\varepsilon\nu.
\end{equation}
The scalars \(a,b\) are fixed by the corrector construction, and only their displayed bounds are used below. For the corrector in \eqref{eq:profile4},
\[
\int_0^1e^{-2z/\sqrt{2\varepsilon\nu}}\,dz
\le\frac{\sqrt{2\varepsilon\nu}}2.
\]
The same estimate holds at the upper plate. Using the bounds for \(a,b\),
and differentiating twice, gives
\begin{equation}\label{eq:B4bounds}
\|\BB_4[w]^h\|
\le c\bigl((\varepsilon\nu)^{3/4}+\varepsilon\nu\bigr)\|w\|,\quad
\|\partial_3^2\BB_4[w]^h\|
\le c(\varepsilon\nu)^{-1/4}\|w\|.
\end{equation}

We further note that divergence freeness gives an improved vertical-component bound
\begin{equation}\label{eq:B3derivative}
\|\partial_3\BB[f]^3\|=\|\divh\BB[f]^h\|
\le c(\varepsilon\nu)^{1/4}\|f\|_{H^1}.
\end{equation}
Finally, we note that the bound on $\BB[f]$ can be improved when tested against a $ W_\sigma^{1,2}$ function. Due to the specific structure, we get 
$$\langle\BB[f],g\rangle
=\int_\OO \mathcal M^h f^h\cdot g^h
 +\int_\OO\mathcal M^z(f^1\partial_2g^3-f^2\partial_1g^3),$$
which implies
\begin{equation}\label{eq:Bdual}
\begin{aligned}
|\langle\BB[f],g\rangle|
&\le c(\varepsilon\nu)^{1/4}\|f\|_{L_h^2}
\bigl(\|g\|+\|\nabla^hg\|\bigr),
\end{aligned}
\end{equation}
applying Cauchy-Schwarz first in the horizontal variables, then in the vertical, and using the $L^2_z$ bound on $\mathcal{M}$.

\subsection{Stochastic Preliminaries} \label{subs stoch prelim}

By a filtered probability space $\mathcal{S}$, we mean a quartet $(\Omega,\mathcal{F},(\mathcal{F}_t), \mathbbm{P})$ satisfying the usual conditions of completeness and right continuity. Let us fix an auxiliary Hilbert Space $\mathfrak{U}$ with orthonormal basis $(e_i)$. We say that $\mathcal{W}$ is a Cylindrical Brownian Motion over $\mathfrak{U}$ with respect to $\mathcal{S}$ if $\mathcal{W}_t = \sum_{i=1}^\infty e_iW^i_t$ as a limit in $L^2(\Omega;\mathfrak{U}')$ for some collection $(W^i)$ of i.i.d. standard real valued Brownian Motions with respect to $\mathcal{S}$, and $\mathfrak{U}'$ an enlargement of the Hilbert Space $\mathfrak{U}$ such that the embedding $J: \mathfrak{U} \rightarrow \mathfrak{U}'$ is Hilbert-Schmidt. In this case $\mathcal{W}$ is a $JJ^*-$Cylindrical Brownian Motion over $\mathfrak{U}'$. Given a process $F:[0,T] \times \Omega \rightarrow \mathscr{L}^2(\mathfrak{U};\mathscr{H})$ progressively measurable and such that $F \in L^2\left(\Omega \times [0,T];\mathscr{L}^2(\mathfrak{U};\mathscr{H})\right)$, for any $0 \leq t \leq T$ we define the stochastic integral $$\int_0^tF_sd\mathcal{W}_s \coloneqq \sum_{i=1}^\infty \int_0^tF_s(e_i)dW^i_s,$$ where the infinite sum is taken in $L^2(\Omega;\mathscr{H})$. We can extend this notion to processes $F$ which are such that $F(\omega) \in L^2\left( [0,T];\mathscr{L}^2(\mathfrak{U};\mathscr{H})\right)$ for $\mathbbm{P}-a.e.$ $\omega$ via the traditional localisation procedure. In this case the stochastic integral is a local martingale in $\mathscr{H}$. We thus consider $\mathcal{G}$ and its relatives as an operator on $\mathfrak{U}$ by $\mathcal{G}(e_i) = \mathcal{G}_i$, see [\cite{goodair2024stochastic}] Subchapter 3.2. We defer to [\cite{goodair2024stochastic}] Chapter 2 for further details on this construction and properties of the stochastic integral. We shall make use of the Burkholder-Davis-Gundy Inequality ([\cite{da2014stochastic}] Theorem 4.36) and the energy identity ([\cite{goodair2024stochastic}] Proposition 4.3, [\cite{liu2015stochastic}] Theorem 4.2.5), as well as the following Stochastic Gr\"{o}nwall Lemma which is only slightly and straightforwardly modified from [\cite{glatt2009strong}] Lemma 5.3.

\begin{lemma} \label{gronny}
    Let $0 < T$ be a fixed time horizon, and let $\boldsymbol{\phi},\boldsymbol{\psi}, \boldsymbol{\eta}$ be real-valued, non-negative stochastic processes. Assume that there exist constants $0 \leq c',\hat{c}$, $\kappa$ such that for $\mathbbm{P}-a.e.$ $\omega$, \begin{equation} \label{boundingronny} \int_0^{T}\boldsymbol{\eta}_s(\omega) ds \leq c'\end{equation} 
     and for all stopping times $0 \leq \theta \leq \theta' \leq T$,
    $$\mathbbm{E}\left(\sup_{r \in [\theta,\theta']}\boldsymbol{\phi}_r\right) + \mathbbm{E}\left(\int_{\theta}^{\theta'}\boldsymbol{\psi}_sds \right)\leq \hat{c}\mathbbm{E}\left(\left(\boldsymbol{\phi}_{\theta} + \kappa \right) + \int_{\theta}^{\theta'} \boldsymbol{\eta}_s\boldsymbol{\phi}_sds\right) < \infty. $$
    Then there exists a constant $C$ dependent only on $c',\hat{c},T$ such that $$\mathbbm{E}\left(\sup_{r \in [0,T]}\boldsymbol{\phi}_r\right) + \mathbbm{E}\left(\int_{0}^{T}\boldsymbol{\psi}_sds\right) \leq C\left[\mathbbm{E}(\boldsymbol{\phi}_{0}) + \kappa\right].$$
    \end{lemma}

\subsection{Transport-Stretching Noise} \label{subs transport stretch}

We now address key properties of the anisotropically scaled transport-stretching noise operator, starting by recalling the definition of $\mathcal{G}$ from (\ref{definition of G}) for vector fields $\phi$, $f$ as
$$\mathcal{G}_{\phi}f = \sum_{j=1}^3\left(\phi^j\partial_jf + f^j\nabla \phi^j\right).$$

\begin{lemma} \label{unnamed}
    Let $\phi \in W^{1,\infty}$ satisfy that $\nabla \cdot \phi = 0$. Then for all $f\in H^1$ and $g \in H^1$ such that either $f = 0$ on $\partial \mathcal{O}$ or $g = 0$ on $\partial \mathcal{O}$, we have that
    \begin{align*}
        \inner{\mathcal{T}_{\phi}f}{g} &= -\inner{f}{\mathcal{T}_{\phi}g},\\
        \inner{\mathcal{S}_{\phi}f}{g} &= \inner{f}{\mathcal{T}_g\phi}.
    \end{align*}
Therefore, we define the operators $\mathcal{T}_{\phi}^*$, $\mathcal{S}_{\phi}^*$ and $\mathcal{G}_{\phi}^*$ on $H^1$ by
\begin{align*}
    \mathcal{T}_{\phi}^*f &= -\mathcal{T}_{\phi}f,\\
    \mathcal{S}_{\phi}^*f &= \mathcal{T}_{f}\phi,\\
    \mathcal{G}_{\phi}^*f &= \mathcal{T}_{\phi}^*f + \mathcal{S}_{\phi}^*f.
\end{align*}
    
\end{lemma}

\begin{proof}
The first result is classical, using a compactly supported approximation of either $f$ or $g$ in $H^1$ to conduct the integration by parts. For the second result, we simply observe that
    \begin{align*}
        \inner{\mathcal{S}_{\phi}f}{g} = \sum_{j=1}^3\sum_{l=1}^3\inner{f^j\partial_l\phi^j}{g^l} = \sum_{j=1}^3\sum_{l=1}^3\inner{f^j}{g^l\partial_l\phi^j} = \sum_{j=1}^3\inner{f^j}{\mathcal{T}_g\phi^j} = \inner{f}{\mathcal{T}_g\phi}.
    \end{align*}
\end{proof}

Let us now fix the spatial correlation functions $(\xi_i)$, $\xi_i: \T^2 \rightarrow \R^3$, satisfying $\nabla \cdot \xi_i = 0$ and $\sum_{i=1}^\infty \norm{\xi_i}_{W^{4,\infty}}^2 < \infty$. This level of smoothness is required for the smoother solutions of the limit equation later obtained in Lemma \ref{strong existence for 2D NS}. We further recall and extend several notations from Subsection \ref{subs structure}, firstly the splitting and scaling 
\begin{align*}
    \xi_i^h = \left(\xi_i^1, \xi_i^2, 0\right), \qquad \xi_i^z = \left(0, 0, \xi_i^3 \right), \qquad \tilde{\xi}_i^z = \nu^{\alpha}\xi_i^z
\end{align*}
with
$$ \tilde{\xi}_i = \xi_i^h + \tilde{\xi}_i^z = \left(\xi_i^1, \xi_i^2, \nu^{\alpha}\xi_i^3\right)$$
and then the corresponding operators
$$\tilde{\mathcal{G}}_i = \mathcal{G}_{\tilde{\xi}_i}, \qquad \tilde{\mathcal{G}}_i^z = \mathcal{G}_{\tilde{\xi}_i^z}. $$
The divergence-free property is preserved for all variants of $\xi_i$ considered. We thus obtain expressions for $\tilde{\mathcal{G}}_i^*$, $\mathcal{G}_i^{h,*}$ and $\tilde{\mathcal{G}}_i^{z,*}$ due to Lemma \ref{unnamed}. Noting that $\mathcal{T}_{\xi_i^h}f = \sum_{j=1}^2\xi_i^j\partial_jf$ and $\mathcal{T}_{\tilde{\xi}_i^z}f = \tilde{\xi}_i^3\partial_3f$, one deduces the following bounds. As we are considering the limit $\nu \rightarrow 0$, we assume $\nu \leq 1$ for simplicity.

\begin{lemma} \label{bounds on split noise}
There exists a constant $c$ such that, for all $f \in H^1$,
\begin{align*}
    \norm{\mathcal{G}_i^hf} + \norm{\mathcal{G}_i^{h,*}f} &\leq c\norm{\xi_i^h}_{W^{1,\infty}}\left(\norm{\nabla^hf} + \norm{f}\right),\\
     \norm{\tilde{\mathcal{G}}_i^zf} + \norm{\tilde{\mathcal{G}}_i^{z,*}f} &\leq c\nu^{\alpha}\norm{\xi_i^z}_{W^{1,\infty}}\left(\norm{\partial_3f} + \norm{f}\right),\\
     \norm{\tilde{\mathcal{G}}_if}+\norm{\tilde{\mathcal{G}}_i^*f}
&\le c\norm{\xi_i}_{W^{1,\infty}}
\left(\norm{\nabla^hf}+\nu^\alpha\|\partial_3f\|+\|f\|\right).
\end{align*}
    
\end{lemma}

In addition, we will need the following control in energy estimates to handle the contributions of the It\^{o}-Stratonovich corrector and quadratic variation.

\begin{lemma}
    For any parameter $0 < \delta$, there exists a constant $c_{\delta}$ such that, for all $f\in H^2 \cap W^{1,2}_{\sigma}$,
        \begin{align}
\inner{\mathcal{P}\tilde{\mathcal{G}}_i\mathcal{P}\tilde{\mathcal{G}}_if}{f} +  \norm{\mathcal{P}\tilde{\mathcal{G}}_if}^2  &\leq 
\norm{\xi_i}_{W^{1,\infty}}^2\left(c_{\delta}\norm{f}^2 + \delta\norm{\nabla^hf}^2 +  \delta\nu^{2\alpha}\norm{\partial_3f}^2 \right) \label{bigbound1NEW}\\
    \inner{\tilde{\mathcal{G}}_if}{f}^2 &\leq c\norm{\xi_i}^2_{W^{1,\infty}}\norm{f}^4. \label{bigbound2NEW}
\end{align}
\end{lemma}

\begin{proof}
    In the direction of (\ref{bigbound1NEW}), we have that
    \begin{align*} \nonumber
\inner{\mathcal{P}\tilde{\mathcal{G}}_i\mathcal{P}\tilde{\mathcal{G}}_if}{f} +  \norm{\mathcal{P}\tilde{\mathcal{G}}_if}^2 &= \inner{\mathcal{P}\tilde{\mathcal{G}}_if}{\tilde{\mathcal{G}}_i^*f + \tilde{\mathcal{G}}_if}\\ &= \inner{\mathcal{P}\tilde{\mathcal{G}}_if}{\mathcal{S}_{\tilde{\xi}_i}^*f + \mathcal{S}_{\tilde{\xi}_i}f}\\ &\leq c\norm{\tilde{\mathcal{G}}_if}\norm{\tilde{\xi}_i}_{W^{1,\infty}}\norm{f}\\
&\leq c\norm{\xi_i}_{W^{1,\infty}}^2\norm{f}\left(\norm{\nabla^hf}+\nu^\alpha\|\partial_3f\|+\|f\| \right)
    \end{align*}
so the result follows by applying Young's Inequality. As for (\ref{bigbound2NEW}), observe that
\begin{align*}
    \inner{\tilde{\mathcal{G}}_if}{f}^2 \leq 2\inner{\mathcal{T}_{\tilde{\xi}_i}f}{f}^2 + 2\inner{\mathcal{S}_{\tilde{\xi}_i}f}{f}^2 \leq 2\inner{\mathcal{T}_{\tilde{\xi}_i}f}{f}^2 + c\norm{\tilde{\xi}_i}^2_{W^{1,\infty}}\norm{f}^4
\end{align*}
from which we simply use the cancellation $\inner{\mathcal{T}_{\tilde{\xi}_i}f}{f} = 0$ to conclude.
\end{proof}
The constants in the above lemmas are of course independent of $\xi_i$ and $\nu$. 

\section{The Main Result}

This section is dedicated to the statement and proof of the main result. Subsection \ref{subs def and state} introduces the key definitions and states the main result, Theorem \ref{main result}. To facilitate the proof, we establish the solution theory and obtain a priori estimates for the limit equation in Subsection \ref{subs a priori limit}. Martingale weak solutions of the stochastic Navier-Stokes-Coriolis equations are constructed in Subsection \ref{subs selection martingale weak}. The main result is then proved in Subsection \ref{subs proof of main}.

\subsection{Definitions and Statement of the Main Result} \label{subs def and state}

For the remainder of this work, we fix an arbitrary time interval $[0,T]$ on which our analysis falls. We recall that the regularity of the spatial correlation functions $(\xi_i)$ was fixed in Subsection \ref{subs transport stretch}. Let us begin by defining the notion of a martingale weak solution of the equation (\ref{expanded ito form}). We shall use the more compact representation (\ref{some main equation 2 Ito}). 

\begin{definition} \label{definitionofspacetimeweakmartingale}
Let $\mathcal{S}$ be a filtered probability space and $u_0 \in L^2_{\sigma}$ be deterministic. A pair $(u, \mathcal{W})$, where $\mathcal{W}$ is a Cylindrical Brownian Motion over $\mathfrak{U}$ with respect to $\mathcal{S}$, and $u$ is a progressively measurable process in $W^{1,2}_{\sigma}$ such that  $u \in  L^{\infty}\left([0,T];L^2_{\sigma}\right) \cap L^2\left([0,T];W^{1,2}_{\sigma}\right)$ $\mathbbm{P}-a.s.$, is said to be a martingale weak solution of (\ref{expanded ito form}) with respect to $\mathcal{S}$ and $u_0$ if the identity
\begin{align*}
     \inner{u_t}{\phi} = \inner{u_0}{\phi} &- \int_0^t\inner{B(u_s, u_s)}{\phi} ds  - \int_0^t \sum_{j=1}^2\inner{\partial_j u_s}{\partial_j \phi} ds - \nu\int_0^t\inner{\partial_3 u_s}{\partial_3\phi} ds\\ & -\int_0^t\inner{ \frac{e_3 \wedge u_s}{\varepsilon}}{\phi} ds +\frac{1}{2} \int_0^t\sum_{i=1}^\infty\inner{\mathcal{P}\tilde{\mathcal{G}}_iu_s}{\tilde{\mathcal{G}}_i^*\phi}ds -  \int_0^t\inner{ \tilde{\mathcal{G}} u_s}{\phi} d\mathcal{W}_s
\end{align*}
holds for every $\phi \in W^{1,2}_{\sigma}$, $\mathbbm{P}-a.s.$ in $\R$, for all $t \in [0,T]$.
\end{definition}


\begin{definition} \label{definitionofstrongsolution2d}
Let $\mathcal{S}$ be a filtered probability space, $w_0 \in W^{1,2}_{\sigma,h}$ be deterministic,  and $\mathcal{W}$ a Cylindrical Brownian Motion over $\mathfrak{U}$ with respect to $\mathcal{S}$ . A process $w$ is said to be a strong solution of (\ref{Ito form for w}) with respect to $\mathcal{S}$, $w_0$ and $\mathcal{W}$ if $w$ is progressively measurable in $W^{2,2}_{\sigma,h}$, $w \in  C\left([0,T];W^{1,2}_{\sigma,h}\right) \cap L^2\left([0,T];W^{2,2}_{\sigma,h}\right)$ $\mathbbm{P}-a.s.$, and satisfies the identity
$$w_t = w_0 - \int_0^tB(w_s, w_s) ds - \int_0^t A_h w_s\, ds  - \sqrt{2\beta}\int_0^tw_s \, ds  - \int_0^t\mathcal{P} \mathcal{G}^h w_s \, d\mathcal{W}_s + \frac{1}{2}\int_0^t\sum_{i=1}^\infty \mathcal{P}\mathcal{G}^h_i\mathcal{G}^h_iw_sds$$
$\mathbbm{P}-a.s.$ in $L^2_{\sigma,h}$, for all $t \in [0,T]$. The solution is said to be unique if for any other solution $v$, $w$ and $v$ are indistinguishable.
\end{definition}

As discussed in Subsection \ref{subs funct anal}, we will consider the solution $w$ as an element of $\bar{W}^{1,2}_{\sigma}$ through its trivial extension. We are now set up to state the main result.

\begin{theorem} \label{main result}
    Let $u_0 \in W^{1,2}_{\sigma,h}$ be deterministic, $\frac{1}{2} < \alpha$, and $(\nu^k)$, $(\varepsilon^k)$ be any two sequences of positive constants converging to zero such that $\left(\frac{\nu^k}{\varepsilon^k} \right)$ converges to $0 \leq \beta < \infty$. Then there exists a filtered probability space $\mathcal{S}$ and a Cylindrical Brownian Motion $\mathcal{W}$ over $\mathfrak{U}$ with respect to $\mathcal{S}$, such that:
    \begin{enumerate}
        \item For every $k \in \N$, there exists a process $u^k$ such that $(u^k,\mathcal{W})$ is a martingale weak solution of (\ref{expanded ito form}) with respect to $\mathcal{S}$ and $u_0$, for $\nu = \nu^k$ and $\varepsilon = \varepsilon^k$;
        \item There exists a unique strong solution $w$ of (\ref{Ito form for w}) with respect to $\mathcal{S}$, $u_0$ and $\mathcal{W}$;
        \item $(u^k) \longrightarrow w$ in $L^2\left(\Omega;L^\infty\left([0,T];L^2_{\sigma}\right)\right)$ as $k \longrightarrow \infty$.
    \end{enumerate}
    
\end{theorem}

We remark that the result continues to hold if $\alpha = \frac{1}{2}$ and $\beta = 0$, though we exclude this case from the proof for simplicity. We comment on adapting the proof to this case after the proof's conclusion.

\subsection{A Priori Estimates for the Limit Equation} \label{subs a priori limit}

Henceforth let us fix the assumptions of Theorem \ref{main result}, that is a deterministic $u_0 \in W^{1,2}_{\sigma,h}$, $\frac{1}{2} < \alpha$, and two sequences of positive constants $(\nu^k)$, $(\varepsilon^k)$ converging to zero such that $\left(\frac{\nu^k}{\varepsilon^k} \right)$ converges to $0 \leq \beta < \infty$. To facilitate a smooth approximation, let us also fix any sequence $(w^m_0)$, $w^m_0 \in W^{2,2}_{\sigma,h}$ convergent to $u_0$ in $W^{1,2}_{\sigma,h}$.

\begin{lemma} \label{strong existence for 2D NS}
    Let $\mathcal{S}$ be a filtered probability space and $\mathcal{W}$ a Cylindrical Brownian Motion over $\mathfrak{U}$ with respect to $\mathcal{S}$. Then there exists a unique strong solution of (\ref{Ito form for w}) with respect to $\mathcal{S}$, $u_0$ and $\mathcal{W}$. Furthermore for any $m$, the unique strong solution of (\ref{Ito form for w}) with respect to $\mathcal{S}$, $w^m_0$ and $\mathcal{W}$ belongs to $C\left([0,T];W^{2,2}_{\sigma,h}\right) \cap L^2\left([0,T];W^{3,2}_{\sigma,h}\right)$ $\mathbbm{P}-a.s.$.
\end{lemma}

\begin{proof}
    In the absence of damping, the existence result was established in [\cite{goodair2024weak}] Theorem 5.4 (stated for the case of Navier boundary conditions, though the proof immediately transfers to the torus) and the smoothness result in [\cite{goodair2026high}] Proposition 4.7. The additional damping does not harm the proofs in any way, so we conclude here. 
\end{proof}

For additional clarity, let us introduce the equations with explicit $k$ dependence:
\begin{align}
    du^k_t = - B(u^k_t, u^k_t) dt - A_h u^k_t\, dt - \nu^k A_{z} u^k_t\, dt  - \mathcal{P} \frac{e_3 \wedge u^k_t}{\varepsilon^k}dt -  \mathcal{P} \tilde{\mathcal{G}} u^k_t \, d\mathcal{W}_t + \frac{1}{2}\sum_{i=1}^\infty \mathcal{P}\tilde{\mathcal{G}}_i\mathcal{P}\tilde{\mathcal{G}}_iu^k_t dt,\label{some main equation 2 Ito k}
\end{align}
as well as 
\begin{equation}
    dw^k_t = - B(w^k_t, w^k_t) dt - A_h w^k_t\, dt  - \sqrt{\frac{2\nu^k}{\varepsilon^k}}w^k_t \, dt  - \mathcal{P} \mathcal{G}^h w^k_t \, d\mathcal{W}_t + \frac{1}{2}\sum_{i=1}^\infty \mathcal{P}\mathcal{G}^h_i\mathcal{G}^h_iw^k_tdt. \label{Ito form for w k}
\end{equation}
The solutions $(w^k)$ will be required in the proof to match the damping to the rotation of $(u^k)$, as in the boundary layer expansion. Therefore $w$ will be approximated in both $k$ for the damping coefficient, and $m$ for the initial condition. The remainder of this subsection addresses the necessary uniform bounds and convergence in these variables, starting with the uniform bounds.

\begin{lemma} \label{uniform bounds on w}
    There exists a constant $C$ such that for any filtered probability space $\mathcal{S}$, Cylindrical Brownian Motion $\mathcal{W}$ and any $k$, the unique strong solution $w^k$ of (\ref{Ito form for w k}) with respect to $\mathcal{S}$, $u_0$ and $\mathcal{W}$ satisfies
    \begin{equation}\label{Estimate1}\sup_{t\in[0,T]}\norm{w^k_t}_{H^1}^2 + \int_0^T\norm{w^k_s}_{H^2}^2ds \leq C\norm{u_0}_{H^1}^2.\end{equation}
Furthermore for any $m$, the unique strong solution $w^{k,m}$ of (\ref{Ito form for w k}) with respect to $\mathcal{S}$, $w^m_0$ and $\mathcal{W}$ satisfies the bound
\begin{equation}
    \label{Estimate2}
    \mathbbm{E}\left[\sup_{t\in[0,T]}\norm{w^{k,m}_t}_{H^2}^2 + \int_0^T\norm{w^{k,m}_s}_{H^3}^2ds \right] \leq C\norm{w^m_0}_{H^2}^2.
\end{equation}
\end{lemma}

\begin{proof}
In the absence of damping, the bound (\ref{Estimate1}) was shown in Subsection 3.1 of [\cite{goodair2025navier}] and we shall defer some of the details to there. Nevertheless, we present the core idea here. Setting $\eta^k \coloneqq \textnormal{curl}w^k$, then the scalar $\eta^k$ satisfies
$$d\eta^k_t = -(w^k_t \cdot \nabla^h)\eta^k_t \, dt + \Delta^h \eta^k_t dt - \sqrt{\frac{2\nu^k}{\varepsilon^k}} \eta^k_t \, dt - \sum_i \mathcal{T}_{\xi_i^h}\eta^k_t dW^i_t + \frac{1}{2}\sum_{i=1}^\infty \mathcal{T}_{\xi_i^h}^2 \eta^k_t \, dt$$
weakly. The pure transport noise provides a perfect cancellation in the $L^2$ energy computation, whilst the damping term can simply be dropped, leading to the inequality
$$\norm{\eta^k_t}^2 + 2\int_0^t\norm{\nabla^h\eta^k_s}^2ds \leq \norm{\textnormal{curl}u_0}^2$$
where inequality, rather than equality, is due to the damping. By taking the supremum in time, and using equivalence of the $H^1$ norm and $L^2$ norm of the curl, the estimate is proven.

The $H^2$ energy bound is somewhat more involved, so we provide a proof. We will drop the $k,m$ superscripts for the calculations in the proof for notational convenience. Using the energy identity, we have
\begin{align*}
    d\|w_t\|_{H^2}^2 &= -2\|\nabla^h w_t\|_{H^2}^2 dt -2 \inner{\PP (w_t \cdot\nabla )w_t}{w_t}_{H^2} dt-2\inner{\sqrt{\frac{2\nu}{\varepsilon}}w_t}{w_t}_{H^2}dt \\
    &-2\inner{\PP \GG^hw_t}{w_t}_{H^2}d\mathcal{W}_t\\
    &+\sum_i \inner{\PP(\Gi^h)^2w_t}{w_t}_{H^2}dt + \sum_i \norm{\PP \Gi^hw_t}^2_{H^2}dt.
\end{align*}
Let us now treat all the terms on the right hand side. In the first line, we keep the negative diffusion term, drop the negative damping term, and estimate the nonlinear term as follows
\begin{align*}
    |\inner{\PP (w_t \cdot\nabla )w_t}{w_t}_{H^2}| &\le \norm{(w_t \cdot\nabla )w_t}_{H^1} \norm{w_t}_{H^3}.
\end{align*}
We now note 
\begin{align*}
    \norm{(w_t \cdot\nabla )w_t}_{H^1} &\lesssim \norm{ w_t}^2_{W^{1,4}} + \norm{w _t\nabla^2w_t} \\ &\lesssim \norm{w_t}_{H^1}\norm{w_t}_{H^2} + \norm{w_t}^{\frac{1}{2}}\norm{w_t}_{H^1}^{\frac{1}{2}} \norm{w_t}_{H^2}^{\frac{1}{2}}\norm{w_t}_{H^3}^{\frac{1}{2}} \lesssim \norm{w_t}_{H^1} \norm{w_t}_{H^2}^{\frac{1}{2}}\norm{w_t}_{H^3}^{\frac{1}{2}},
\end{align*}
which gives 
\begin{align*}
    |\inner{\PP (w_t \cdot\nabla )w_t}{w_t}_{H^2}| &\le c \norm{w_t}_{H^1} \norm{w_t}_{H^2}^{\frac{1}{2}}\norm{w_t}_{H^3}^{\frac{3}{2}} \le c\norm{w_t}_{H^2}^{\frac{1}{2}}\norm{w_t}_{H^3}^{\frac{3}{2}} \\
    \le \frac{1}{2}\norm{w_t}_{H^3}^2 + c\norm{w_t}_{H^2}^2,
\end{align*}
using that the $H^1$ norm of $w$ is bounded uniformly in time by a deterministic constant, and a Young's inequality in the last line.

Next we treat the martingale term using the BDG inequality, and the It\^o-Stratonovich correction term along with the quadratic variation term. For this purpose we recall the bounds:
\begin{align}
    \inner{\PP\GG_i^2f}{f}_{H^2} +  \norm{\PP\GG_if}_{H^2}^2  &\leq c\norm{\xi_i}_{W^{4,\infty}}^2\norm{f}_{H^2}^2, \label{bigbound1} \\
    \inner{\PP\GG_i f}{f}_{H^2}^2 &\leq c\norm{\xi_i}^2_{W^{3,\infty}}\norm{f}^4_{H^2}, \label{bigbound2} 
\end{align}
which were proved in [\cite{goodair2025closed}].
For the martingale term, we have:
\begin{align*}
    \E \left[ \sup_{s\in[0,t]}\left| \int_0^s\inner{\PP \GG^hw_r}{w_r}_{H^2}d\mathcal{W}_r \right|\right] &\le c \E \left[ \left( \int_0^t \left|\inner{\PP \GG^hw_s}{w_s}_{H^2}\right|^2ds \right)^{\frac{1}{2}} \right] \\
    &\le c \E \left[ \left( \int_0^t \norm{w_s}^4_{H^2} ds \right)^{\frac{1}{2}} \right] \\
    &\le c \E \left[\sup_{s\in [0,t]}\norm{w_s}_{H^2} \left( \int_0^t \norm{w_s}^2_{H^2} ds \right)^{\frac{1}{2}} \right] \\
    &\le \frac{1}{2} \E \left[ \sup_{s\in [0,t]}\norm{w_s}_{H^2}^2 \right] + c \E \left[\int_0^t \norm{w_s}^2_{H^2} ds \right].
\end{align*}
For the last line, we simply apply \eqref{bigbound1}
\begin{align*}
    \sum_i \left(\inner{\PP(\Gi^h)^2w_t}{w_t}_{H^2} +  \norm{\PP \Gi^hw_t}^2_{H^2} \right) \le c \norm{w_t}_{H^2}^2.
\end{align*}
Collecting all the estimates, taking the supremum and expectations in the energy identity, we have:
\begin{align*}
    \E \left[ \sup_{t\in[0,t']}\norm{w_t}_{H^2}^2 \right] + \E \left[ \int_0^{t'}\norm{w_t}_{H^3}^2 dt \right] &\le c \E\left[\norm{w_0}_{H^2}^2\right] + c\int_0^{t'} \E\left[ \sup_{s\in[0,t]}\norm{w_s}_{H^2}^2 \right]dt.
\end{align*}
Now an application of the standard Gr\"onwall inequality gives
\begin{equation*}
    \E \left[ \sup_{t\in[0,T]}\norm{w_t}_{H^2}^2 \right] + \E \left[ \int_0^{T}\norm{w_t}_{H^3}^2 dt \right] \le c \norm{w_0}_{H^2}^2,
\end{equation*}
which is exactly the desired result.
\end{proof}

We now address convergence in the variables $k$ and $m$ with two lemmas. These lemmas are proved together. 

\begin{lemma} \label{damping convergence}
    Let $\mathcal{S}$ be a filtered probability space and $\mathcal{W}$ a Cylindrical Brownian Motion over $\mathfrak{U}$ with respect to $\mathcal{S}$. For every $k$, let $w^k$ be the unique strong solution of (\ref{Ito form for w k}) with respect to $\mathcal{S}$, $u_0$ and $\mathcal{W}$, while $w$ is the unique strong solution of (\ref{Ito form for w}) with respect to $\mathcal{S}$, $u_0$ and $\mathcal{W}$. Then $(w^k) \longrightarrow w$ in $L^2\left(\Omega;L^\infty\left([0,T];L^2_{\sigma, h}\right)\right)$ as $k \longrightarrow \infty$.
\end{lemma}

\begin{lemma} \label{ic convergence}
   Let $\mathcal{S}$ be a filtered probability space and $\mathcal{W}$ a Cylindrical Brownian Motion over $\mathfrak{U}$ with respect to $\mathcal{S}$. For every $k$, let $w^{k,m}$ be the unique strong solution of (\ref{Ito form for w k}) with respect to $\mathcal{S}$, $w^m_0$ and $\mathcal{W}$, while $w^k$ is the unique strong solution of (\ref{Ito form for w k}) with respect to $\mathcal{S}$, $u_0$ and $\mathcal{W}$. Then $(w^{k,m}) \longrightarrow w^k$ in $L^2\left(\Omega;L^\infty\left([0,T];L^2_{\sigma, h}\right)\right)$ as $m \longrightarrow \infty$.
\end{lemma}

\begin{proof}
    The proof follows by establishing continuity in the damping parameter and initial data through a Gr\"onwall argument. To ease the notations, consider $w$ to be the solution to
    \begin{equation*}
        dw_t= -A_hw_t dt - B(w_t,w_t)dt - \beta w_tdt - \PP\GG^h w_t d\mathcal{W}_t + \frac{1}{2}\sum_i \PP(\Gi ^h)^2w_t dt,
    \end{equation*}
    with initial data $w_0\in W^{1,2}_{\sigma, h}$, and let $\tilde w$ solve the same equation with parameter $\tilde \beta$ and initial data $\tilde w_0$. Let further $v:= w-\tilde w$. Then the difference process satisfies the equation:
    \begin{equation*}
        dv_t = -A_h v_t dt - B(w_t,v_t)dt - B(v_t,\tilde w_t)dt -\beta v dt-(\beta - \tilde \beta)\tilde w dt -\PP\GG^h v_t d\mathcal{W}_t + \frac{1}{2}\sum_i \PP(\Gi ^h)^2v_t dt,
    \end{equation*}
    with $v_0=w_0-\tilde w_0$. The energy identity now yields:
    \begin{align*}
        d\norm{v_t}^2 = &-2\norm{\nabla^h v_t}^2 dt -2 \inner{B(v_t,\tilde w_t)}{v_t}dt-2\beta \norm{v_t}^2dt -2(\beta - \tilde \beta)\inner{\tilde w_t}{v_t}dt\\
        & -2 \inner{\PP \GG^hv_t}{v_t}d\mathcal{W}_t\\
        &+ \sum_i \left( \inner{\PP(\Gi ^h)^2v_t}{v_t} + \norm{\PP\Gi^hv_t}^2\right)dt.
    \end{align*}
    For the terms in the first line, we drop the $v$ damping term, and estimate the rest as follows:
    \begin{equation*}
        \left|\inner{B(v_t,\tilde w_t)}{v_t} \right| \le c\norm{\nabla^h \tilde w_t}\norm{v_t}_{L^4}^2 \le c \norm{v_t}\norm{v_t}_{H^1}\le \frac{1}{2}\norm{v_t}_{H^1}^2 + c\norm{v_t}^2,
    \end{equation*}
    where we have exploited the deterministic, uniform in time bound on $\|\tilde w\|_{H^1}$ from \ref{uniform bounds on w}, and a $2$-d Ladyzhenskaya inequality. For the remaining term we have
    \begin{equation*}
        \left|(\beta - \tilde \beta)\inner{\tilde w_t}{v_t} \right|\le |(\beta - \tilde \beta)|^2 + c\|v_t\|^2,
    \end{equation*}
    by Cauchy-Schwarz followed by a Young's inequality.\\

    The martingale term, and the remaining noise terms are dealt with exactly the same as the ones in the proof of lemma \ref{uniform bounds on w}, using the corresponding bounds:
    \begin{align*}
    \inner{\PP\GG_i^2f}{f} +  \norm{\PP\GG_if}_{}^2  &\leq c\norm{\xi_i}_{W^{2,\infty}}^2\norm{f}_{}^2,\\
    \inner{\PP\GG_i f}{f}_{}^2 &\leq c\norm{\xi_i}^2_{W^{1,\infty}}\norm{f}^4_{},
    \end{align*}
    which were also proved in [\cite{goodair2025closed}]. For the martingale term, we have:
    \begin{align*}
        \E \left[ \sup_{s\in[0,t]}\left| \int_0^s\inner{\PP \GG^hv_r}{v_r}d\mathcal{W}_r \right|\right] &\le c \E \left[ \left( \int_0^t \left|\inner{\PP \GG^hv_s}{v_s}\right|^2ds \right)^{\frac{1}{2}} \right] \\
        &\le c \E \left[ \left( \int_0^t \norm{v_s}^4 ds \right)^{\frac{1}{2}} \right] \\
        &\le c \E \left[\sup_{s\in [0,t]}\norm{v_s} \left( \int_0^t \norm{v_s}^2 ds \right)^{\frac{1}{2}} \right] \\
        &\le \frac{1}{2} \E \left[ \sup_{s\in [0,t]}\norm{v_s}^2 \right] + c \E \left[\int_0^t \norm{v_s}^2 ds \right],
    \end{align*}
    and for the remaining ones
    \begin{align*}
        \sum_i \left(\inner{\PP(\Gi^h)^2v_t}{v_t} +  \norm{\PP \Gi^hv_t}^2 \right) \le c \norm{v_t}^2.
    \end{align*}
    Collecting all the estimates, we have:
    \begin{align*}
        \E \left[ \sup_{t\in[0,t']}\norm{v_t}^2 \right]  &\le c \E\left[\norm{v_0}^2\right] + c|\beta - \tilde \beta|^2 + c\int_0^{t'} \E\left[ \sup_{s\in[0,t]}\norm{v_s}^2 \right]dt.
    \end{align*}
    Applying Gr\"onwall's lemma, we obtain
    \begin{equation*}
        \E \left[ \sup_{t\in[0,T]}\norm{v_t}^2 \right] \le c\left(\E\left[\norm{v_0}^2\right] + |\beta - \tilde \beta|^2\right),
    \end{equation*}
    which proves both lemmas.
\end{proof}
\subsection{Selection of Martingale Weak Solutions} \label{subs selection martingale weak}
\begin{proposition} \label{prop choosing spaces}
There exists a filtered probability space $\mathcal{S}$ and Cylindrical Brownian Motions $\mathcal{W}$, $\left(\mathcal{W}^n\right)$ over $\mathfrak{U}$ with respect to $\mathcal{S}$, such that:
\begin{enumerate}
    \item For every $k$, there exists a process $u^k$ such that $(u^k,\mathcal{W})$ is a martingale weak solution of (\ref{some main equation 2 Ito k}) with respect to $\mathcal{S}$ and $u_0$;
    \item \label{item2} $u^k$ is obtained as the $n \rightarrow \infty$ limit of Galerkin Approximations $u^{k,n}$ satisfying
    \begin{align} \nonumber
    u^{k,n}_t &= \mathcal{P}_nu_0 - \int_0^t\mathcal{P}_nB(u^{k,n}_s, u^{k,n}_s) ds - \int_0^t \mathcal{P}_n A_h u^{k,n}_s\, ds - \nu^k\int_0^t \mathcal{P}_n A_{z} u^{k,n}_s\, ds \\ & - \int_0^t\mathcal{P}_n\mathcal{P} \frac{e_3 \wedge u^{k,n}_s}{\varepsilon^k}ds-  \int_0^t \mathcal{P}_n\mathcal{P} \tilde{\mathcal{G}} u^{k,n}_s \, d\mathcal{W}^n_s + \frac{1}{2}\int_0^t\sum_{i=1}^\infty \mathcal{P}_n\mathcal{P}\tilde{\mathcal{G}}_i\mathcal{P}\tilde{\mathcal{G}}_iu^{k,n}_s ds. \nonumber
\end{align}
Precisely, $(u^{k,n}) \rightarrow u^k$ in the weak* topology of $L^2\left(\Omega; L^\infty\left([0,T];L^2_{\sigma}\right) \right)$ and the weak topology of $L^2\left(\Omega; L^2\left([0,T];W^{1,2}_{\sigma}\right) \right)$;
\item \label{item2.5}
The solutions $(u^{k,n})$ above, enjoy an additional uniform (in $n$) bound on \\ $L^4\left(\Omega; L^\infty\left([0,T];L^2_{\sigma}\right) \cap L^2\left([0,T];W^{1,2}_{\sigma} \right) \right)$;
\item \label{item3} For every $k$ and $m$, the unique strong solution $w^{k,m}$ of (\ref{Ito form for w k}) with respect to $\mathcal{S}$, $w^m_0$ and $\mathcal{W}$ is the $n \rightarrow \infty$ limit of the unique strong solutions $(w^{k,m,n})$ of (\ref{Ito form for w k}) with respect to $\mathcal{S}$, $w^m_0$ and $(\mathcal{W}^n)$: convergence holds in the topology of $L^2\left(\Omega; C\left([0,T];W^{1,2}_{\sigma,h}\right) \cap L^2\left([0,T];W^{2,2}_{\sigma,h}\right) \right)$. 
\end{enumerate}

\end{proposition}

\begin{proof}
The existence of a martingale weak solution to (\ref{some main equation 2 Ito k}), as a limit of Galerkin Approximations, was given in the isotropic case and without rotation on a bounded domain in [\cite{goodair2024weak}] Theorem 5.1. Anisotropy and alteration of the domain does not disturb the proof in a significant way, as one only suffers the worse bound (\ref{bigbound1NEW}) which is entirely sufficient for the proof in [\cite{goodair2024weak}] by combining with the viscous term. Likewise, the rotational term is a bounded linear function of the solution and does not affect the proof. \\

Let us now address the uniform (in $n$) bound in $L^4\left(\Omega; L^\infty\left([0,T];L^2_{\sigma}\right) \cap L^2\left([0,T];W^{1,2}_{\sigma} \right) \right)$. This is fairly standard, but we sketch the argument for completeness. Applying the energy identity, we deduce
\begin{align*}
    \norm{u^{k,n}_t}^2 + \int_0^t\norm{\nabla^hu^{k,n}_s}^2 + \nu_k\norm{\partial_3^2u^{k,n}_s}^2 ds \le &c\norm{u^{k,n}_0}^2 + c\left| \int_0^t \inner{\mathcal{P}_n\mathcal{P} \tilde{\mathcal{G}} u^{k,n}_s}{u^{k,n}_s} d\mathcal{W}^n_s \right| +\\ &+ c \int_0^t \sum_i \inner{\mathcal{P}_n\mathcal{P}\tilde{\mathcal{G}}_i\mathcal{P}\tilde{\mathcal{G}}_iu^{k,n}_s}{u^{k,n}_s} + \norm{\mathcal{P}_n\mathcal{P} \tilde{\mathcal{G}}_i u^{k,n}_s}^2 ds.
\end{align*}
Using the bound, \eqref{bigbound1NEW} and taking $\delta$ small enough, we see
\begin{align*}
    \norm{u^{k,n}_t}^2 + \int_0^t\norm{\nabla^hu^{k,n}_s}^2 + \nu_k \norm{\partial_3^2u^{k,n}_s}^2 ds \le &c\norm{u^{k,n}_0}^2 + c\left| \int_0^t \inner{\mathcal{P}_n\mathcal{P} \tilde{\mathcal{G}} u^{k,n}_s}{u^{k,n}_s} d\mathcal{W}^n_s \right| +\\ &+ c_\delta \int_0^t \norm{u^{k,n}_s}^2 ds.
\end{align*}
We now take supremums in time, square the inequality and take expectations to obtain (with some elementary bounds and an application of the BDG inequality at the last step)
\begin{align*}
    \E \left[ \left(\sup_{t'\le t}\norm{u^{k,n}_{t'}}^2 +  \int_0^t\norm{\nabla^hu^{k,n}_s}^2 + \nu_k \norm{\partial_3^2u^{k,n}_s}^2 ds\right)^2 \right] \le &c\E \left[ \norm{u^{k,n}_0}^4 \right] \\ &+c \E \left[ \sup_{t' \le t}\left| \int_0^{t'} \inner{\mathcal{P}_n\mathcal{P} \tilde{\mathcal{G}} u^{k,n}_s}{u^{k,n}_s} d\mathcal{W}^n_s \right|^2 \right] \\ &+ c_\delta \E \left[ \int_0^t \norm{u^{k,n}_s}^4 ds \right] \\
    &\le c\E \left[ \norm{u^{k,n}_0}^4 \right] + c  \int_0^t \E \left[ \sup_{t'\le s}\norm{u^{k,n}_{t'}}^4 \right] ds.
\end{align*}
Gr\"onwall's lemma now implies
\begin{equation*}
    \E \left[ \left( \sup_{t\le T}\norm{u^{k,n}_t}^2  + \int_0^T\norm{\nabla^hu^{k,n}_s}^2 + \nu_k \norm{\partial_3^2u^{k,n}_s}^2 ds\right)^2 \right] \le c\E \left[ \norm{u^{k,n}_0}^4 \right],
\end{equation*}
where the constant $c$ is uniform both in $k,n$, and due to the uniform boundedness of the initial data for the Galerkin equations, this establishes the desired estimate.\\

The remaining point of this proposition is that the martingale weak solution $(u^k, \mathcal{W})$ is constructed so that $\mathcal{W}$ is independent of $k$, as well as the convergence $(w^{k,m,n}) \longrightarrow w^{k,m}$.
The strategy allowing to achieve this, is presented in [\cite{breit-hofmanova2016}].
The first step is to fix a probability space with a Cylindrical Brownian Motion $\mathcal{W}$. Then, we construct the probabilistically strong solution to the equations for $u^{k,n},w^{k,m}$ driven by this $\mathcal{W}$. Tightness is a consequence of the standard uniform estimates used in the proof referenced above. Now the key idea is to look at the following collection of variables indexed by $n$:
\begin{equation*}
    \left\{ \left( u^{k,n},w^{k,m}, \mathcal{W} \right)_{k,m\in \mathbb{N}}\,|\, n\in \mathbb{N} \right\}.
\end{equation*}
Passing to a subsequence, which we don't relabel, we get that this family converges in law to the family $(u^k,w^{k,m}, \mathcal{W})_{k,m \in \mathbb{N}}$. Applying the Skorokhod embedding theorem, we now get the family
\begin{equation*}
    \left\{ \left( \tilde{u}^{k,n},\tilde{w}^{k,m,n}, \tilde{\mathcal{W}}^n \right)_{k,m\in \mathbb{N}}\,|\, n\in \mathbb{N} \right\}
\end{equation*}
converging almost surely to $(\tilde u^k, \tilde w^{k,m},\tilde{ \mathcal{W}})_{k,m \in \mathbb{N}}$. The topologies in which this happens are exactly the ones implied by the available uniform bounds leading to tightness for $u^{k,n}$. For $\tilde w^{k,m,n}$, we get almost sure convergence in $C\left([0,T];W^{2,2}_{\sigma,h}\right) \cap L^2\left([0,T];W^{3,2}_{\sigma,h}\right)$ due to the regularity of the initial condition, as well as  $L^2\left(\Omega; C\left([0,T];W^{1,2}_{\sigma,h}\right) \cap L^2\left([0,T];W^{2,2}_{\sigma,h}\right) \right)$ due to the deterministic bound from Lemma \ref{uniform bounds on w} allowing to extend to the probabilistic $L^2$ mode of convergence. The identification of the equations solved by the variables is completely standard, so we do not reproduce it. From now on, we work on this common probability space, and drop the tildes. 
\end{proof}

\subsection{Proof of the Main Result} \label{subs proof of main}

\begin{proof}[Proof of Theorem \ref{main result}]
    We choose the filtered probability space $\mathcal{S}$ and Cylindrical Brownian Motions $\mathcal{W}$, $(\mathcal{W}^n)$ as specified in Proposition \ref{prop choosing spaces}. Furthermore, we introduce a sequence of more regular initial conditions $(w^m_0)$, $w^m_0 \in W^{2,2}_{\sigma,h}$ which converges to $u_0$ in $W^{1,2}_{\sigma,h}$. Let us now fix some notation for the solutions and their approximations, as required in the proof:
    \begin{itemize}
        \item $(u^k)$ and $(u^{k,n})$ are as in Proposition \ref{prop choosing spaces};
        \item $w$ is the unique strong solution of (\ref{Ito form for w}) with respect to $\mathcal{S}$, $u_0$ and $\mathcal{W}$;
        \item $(w^k)$ are the unique strong solutions of (\ref{Ito form for w k}) with respect to $\mathcal{S}$, $u_0$ and $\mathcal{W}$;
        \item $(w^{k,m})$ are the unique strong solutions of (\ref{Ito form for w k}) with respect to $\mathcal{S}$, $(w^m_0)$ and $\mathcal{W}$;
        \item $(w^{k,m,n})$ are the unique strong solutions of (\ref{Ito form for w k}) with respect to $\mathcal{S}$, $(w^m_0)$ and $(\mathcal{W}^n)$.
    \end{itemize}
We begin the proof by establishing sufficiency of working at the level of the  sequences. We have that
    \begin{align*}
        \mathbbm{E}\left[\norm{u^k - w}_{L^\infty([0,T];L^2_{\sigma})}^2 \right] &\leq 2 \mathbbm{E}\left[\norm{u^k - w^k }_{L^\infty([0,T];L^2_{\sigma})}^2\right] + 2\mathbbm{E}\left[\norm{w^k - w }_{L^\infty([0,T];L^2_{\sigma})}^2\right]
    \end{align*}
so by Lemma \ref{damping convergence}, 
\begin{align*}
       \lim_{k 
       \rightarrow \infty} \mathbbm{E}\left[\norm{u^k - w}_{L^\infty([0,T];L^2_{\sigma})}^2 \right] &\leq 2\lim_{k \rightarrow \infty} \mathbbm{E}\left[\norm{u^k - w^k }_{L^\infty([0,T];L^2_{\sigma})}^2\right].
    \end{align*}
Moreover,
$$\mathbbm{E}\left[\norm{u^k - w^k }_{L^\infty([0,T];L^2_{\sigma})}^2\right] \leq 2\mathbbm{E}\left[\norm{u^k - w^{k,m} }_{L^\infty([0,T];L^2_{\sigma})}^2\right] + 2\mathbbm{E}\left[\norm{w^{k,m} - w^k }_{L^\infty([0,T];L^2_{\sigma})}^2\right]$$
and due to the weak* convergence specified in item \ref{item2} and the strong convergence in \ref{item3} of Proposition \ref{prop choosing spaces},
$$\mathbbm{E}\left[\norm{u^k - w^{k,m} }_{L^\infty([0,T];L^2_{\sigma})}^2\right] \leq \liminf_{n \rightarrow \infty}\mathbbm{E}\left[\norm{u^{k,n} - w^{k,m,n} }_{L^\infty([0,T];L^2_{\sigma})}^2\right].$$
Therefore, we have that
\begin{align*}
    &\mathbbm{E}\left[\norm{u^k - w^k }_{L^\infty([0,T];L^2_{\sigma})}^2\right]\\
    &\leq 2\lim_{m \rightarrow \infty} \liminf_{n \rightarrow \infty}\mathbbm{E}\left[\norm{u^{k,n} - w^{k,m,n} }_{L^\infty([0,T];L^2_{\sigma})}^2\right] + 2\lim_{m \rightarrow \infty}\mathbbm{E}\left[\norm{w^{k,m} - w^k }_{L^\infty([0,T];L^2_{\sigma})}^2\right]\\
    &= 2\lim_{m \rightarrow \infty} \liminf_{n \rightarrow \infty}\mathbbm{E}\left[\norm{u^{k,n} - w^{k,m,n} }_{L^\infty([0,T];L^2_{\sigma})}^2\right]
\end{align*}
having applied Lemma \ref{ic convergence}. In total,
\begin{equation} \label{in total}
   \lim_{k 
       \rightarrow \infty} \mathbbm{E}\left[\norm{u^k - w}_{L^\infty([0,T];L^2_{\sigma})}^2 \right] \leq 4\lim_{k 
       \rightarrow \infty}\lim_{m \rightarrow \infty}\liminf_{n \rightarrow \infty}\mathbbm{E}\left[\norm{u^{k,n} - w^{k,m,n} }_{L^\infty([0,T];L^2_{\sigma})}^2\right]
\end{equation}
so the computation boils down to treating $\norm{u^{k,n} - w^{k,m,n} }^2$. At this point we will drop the superscripts, thus renaming the variables:
\begin{equation*}
    u^{k,n} \mapsto u, \quad w^{k,m,n} \mapsto w, \quad \mathcal{W}^n \mapsto \mathcal{W},
\end{equation*}
and the dependence on parameters will be tracked carefully in the constants, to ensure the correct dependencies. We would like to carry out an energy estimate for the difference, but in order to do this, we need to look at the difference corrected by the boundary corrector. We thus define
$$v:= u-w-\mathcal{B}[w].$$
To be able to write down the energy identity for $v$, we need to first understand the semimartingale decomposition of $\mathcal{B}[w]$. Due to the specific form \eqref{eq:corrector-def}, the deterministic linear map $\mathcal{B}[\cdot]$ commutes with the horizontal derivatives. In combination with $w$ only depending on the horizontal variables, we have
\begin{equation}\label{eq:dB}
\begin{aligned}
d\BB={}&-\BB[B(w,w)]\,dt-A_h\BB\,dt
-\sqrt{\frac{2\nu}{\varepsilon}}\BB\,dt\\
&+\frac12\sum_i\BB[\PP(\GG_i^h)^2w]\,dt
-\sum_i\BB[\PP\GG_i^hw]\,dW^{i}.
\end{aligned}
\end{equation}
With this decomposition in place, we can apply the energy identity to $v$, understood as $u-w- \mathcal{B}[w]$:
\begin{align}
d\|v\|^2=&
-2\langle\PP_nB(u,u)-B(w,w)-\BB[B(w,w)],v\rangle\,dt\nonumber\\
&-2\langle\PP_nA_hu-A_hw-A_h\BB,v\rangle\,dt\nonumber\\
&-2\nu\langle\PP_nA_zu,v\rangle\,dt
-\frac2\varepsilon\langle\PP_n\PP(e_3\wedge u),v\rangle\,dt\nonumber\\
&+2\sqrt{\frac{2\nu}{\varepsilon}}\langle w+\BB,v\rangle\,dt\nonumber\\
&+\sum_i\langle
\PP_n\PP\Gt_i\PP\Gt_iu-\PP(\GG_i^h)^2w-\BB[\PP(\GG_i^h)^2w],v\rangle\,dt
\nonumber\\
&+\sum_i\|\PP_n\PP\Gt_iu-\PP\GG_i^hw-\BB[\PP\GG_i^hw]\|^2\,dt
\nonumber\\
&-2\sum_i\langle
\PP_n\PP\Gt_iu-\PP\GG_i^hw-\BB[\PP\GG_i^hw],v\rangle\,dW^{i}.
\label{eq:exact}
\end{align}
We now wish to bound the right hand side by terms that can be treated by an application of a (stochastic) Gr\"onwall lemma \ref{gronny}. This is somewhat delicate, as one has to exploit various exact cancellations permitted by the specific boundary layer corrector, as well as top order derivative cancellation between the It\^o-Stratonovich correction and the quadratic variation term.\\

\noindent In order to see the above features, we need to isolate the error coming from the Galerkin projections. To this end we note
$$\PP_nv=v+(I-\PP_n)(w+\BB),$$ and consequently
$$\langle\PP_n f,v\rangle
=\langle f,v\rangle+
\langle f,(I-\PP_n)(w+\BB)\rangle.$$
For the nonlinear term, expand \(u=v+w+\BB\):
$$ B(u,u)-B(w,w)
=B(u,v)+B(v,w)+B(\BB,w)+B(u,\BB),$$
which (exploiting the skew-symmetry of $B$) gives
\begin{equation} \label{eq:galerkin-res-1}
    \langle\PP_nB(u,u)-B(w,w),v\rangle
=\langle B(v,w)+B(\BB,w)+B(u,\BB),v\rangle\\
+\langle B(u,u),(I-\PP_n)(w+\BB)\rangle.
\end{equation}
In both the horizontal and vertical viscous terms we can drop the projection, as we can choose the basis for the Stokes operator, that also diagonalises the horizontal Laplacian. In this case:
$$\langle\PP_nA_hu-A_hw-A_h\BB,v\rangle
=\|\nabla^hv\|^2,$$as well as
$$\langle\PP_nA_zu,v\rangle
=\|\partial_3v\|^2+\langle A_z\BB,v\rangle.$$
As for the Coriolis term, we get
\begin{equation}\label{eq:rotation-exact}
\langle\PP_n\PP(e_3\wedge u),v\rangle
=\langle e_3\wedge\BB,v\rangle
+\langle e_3\wedge u,(I-\PP_n)(w+\BB)\rangle,
\end{equation}
as $\langle e_3\wedge v,v\rangle=0$ due to skew symmetry, and $\langle e_3\wedge w,v\rangle=0$ due to a Taylor-Proudman theorem asserting that $ e_3 \wedge f$ is a gradient if and only if the (mean free) $f$ depends only on the horizontal variables.\\
\noindent
In the It\^o-Stratonovich correction term, we note that $ \PP\Gt_i\PP\Gt_iu = \PP\Gt_i\PP\Gt_i v + \PP (\GG^h_i)^2 w + \PP\Gt_i\PP\Gt_i \mathcal{B} $, so that we have
\begin{equation}\label{eq:ito-exact}
\begin{aligned}
&\langle
\PP_n\PP\Gt_i\PP\Gt_iu-\PP(\GG_i^h)^2w-\BB[\PP(\GG_i^h)^2w],v\rangle=\\
&\quad=
\langle\PP\Gt_iv,\Gt_i^*v\rangle+\langle\PP\Gt_i\BB,\Gt_i^*v\rangle
-\langle\BB[\PP(\GG_i^h)^2w],v\rangle +
\langle\PP\Gt_iu,\Gt_i^*(I-\PP_n)(w+\BB)\rangle.
\end{aligned}
\end{equation}
For the quadratic variation term, we note that
\begin{equation*}\label{eq:diffusion-split}
\begin{aligned}
&\PP_n\PP\Gt_iu-\PP\GG_i^hw-\BB[\PP\GG_i^hw]\\
&\quad=
\PP_n\PP\bigl(\Gt_iv+\Gt_i\BB-\BB[\PP\GG_i^hw]\bigr)-
(I-\PP_n)\bigl(\PP\GG_i^hw+\BB[\PP\GG_i^hw]\bigr),
\end{aligned}
\end{equation*}
where the two terms are orthogonal. This along with bounding the norm of $\PP_n$ by one, gives
\begin{equation}\label{eq:qv}
\begin{aligned}
&\|\PP_n\PP\Gt_iu-\PP\GG_i^hw-\BB[\PP\GG_i^hw]\|^2\\
&\quad\le
\|\PP\Gt_iv+\PP\Gt_i\BB-\BB[\PP\GG_i^hw]\|^2+
\|(I-\PP_n)(\PP\GG_i^hw+\BB[\PP\GG_i^hw])\|^2.
\end{aligned}
\end{equation}
Finally, for the integrand of the martingale term we use that 
\begin{equation}\label{eq:mart-split}
\begin{aligned}
&\langle\PP_n\PP\Gt_iu-\PP\GG_i^hw-\BB[\PP\GG_i^hw],v\rangle\\
&\quad=
\langle\Gt_iv+\Gt_i\BB-\BB[\PP\GG_i^hw],v\rangle+\langle\Gt_iu,(I-\PP_n)(w+\BB)\rangle.
\end{aligned}
\end{equation}
Let us now collect all the Galerkin projection remainder terms, that is the terms from equations \eqref{eq:galerkin-res-1}, \eqref{eq:rotation-exact}, \eqref{eq:ito-exact}, \eqref{eq:qv} and \eqref{eq:mart-split} into $R_i$ for $i \in \{1,2,3,4,5\}$, where $R_5$ is an It\^o integrand.
We can now update the energy balance to
\begin{align}\label{eq:EE_to_est}
    d\|v\|^2\le& -2\langle B(v,w)+B(\BB,w)+B(u,\BB) -\BB[B(w,w)],v\rangle dt\nonumber\\
    &-2 \|\nabla^hv\|^2 dt\nonumber\\
    &-2\nu \|\partial_3v\|^2dt -2\nu\langle A_z\BB,v\rangle dt -\frac{2}{\varepsilon} \langle e_3\wedge\BB,v\rangle dt\nonumber\\
    &+2\sqrt{\frac{2\nu}{\varepsilon}}\langle w+\BB,v\rangle\,dt\nonumber\\ 
    &+\sum_i \left( \langle\PP\Gt_iv,\Gt_i^*v\rangle dt+\langle\PP\Gt_i\BB,\Gt_i^*v\rangle dt-\langle\BB[\PP(\GG_i^h)^2w],v\rangle dt \right) \nonumber \\
    &+\sum_i\|\PP\Gt_iv+\PP\Gt_i\BB-\BB[\PP\GG_i^hw]\|^2 dt \nonumber \\
    &-2 \sum_i\langle\Gt_iv+\Gt_i\BB-\BB[\PP\GG_i^hw],v\rangle dW^i \nonumber\\
    &+ \sum_{i=1}^4 R_i \, dt + R_5 d\mathcal{W}_t.
\end{align}
We are now in a good setup to start estimating terms on the right hand side "group by group". We start with the nonlinear terms  i.e. the first line of \eqref{eq:EE_to_est}, and treat each term separately.\\

\noindent For the first term, we use \eqref{aux-est2} componentwise, with $f=v$ and $g=\partial w$, followed by Young's inequality ($\delta$ to be chosen later, and $\|w\|_{H^1}$ absorbed into the absolute constant due to the deterministic bound \eqref{Estimate1}):
\begin{equation}\label{eq:Bvw}
|\langle B(v,w),v\rangle| \le c \|w\|_{H^1}\|v\| \|\nabla^hv\|
\le\delta\|\nabla^hv\|^2+c_{\delta}\|v\|^2.
\end{equation}
The second term is treated by a horizontal H\"older with exponents $4,2,4$ followed by a vertical Cauchy-Schwarz and an application of \eqref{eq:ladyzh}:
\begin{equation}\label{eq:BBw}
\begin{aligned}
|\langle B(\BB,w),v\rangle|
&\le c\|\mathcal M^h\|_{L_z^2}
\|w\|_{L_h^4}\|\nabla^hw\|_{L_h^2}\|v\|_{L_z^2L_h^4}\\
&\le c(\varepsilon\nu)^{1/4}(\|v\|\|\nabla^hv\|)^{\frac{1}{2}}\le c(\varepsilon\nu)^{1/4}(\|v\| + \|\nabla^hv\|)\\
&\le \delta \|\nabla^h v\|^2 + \|v\|^2 + c_\delta (\varepsilon \nu)^{\frac{1}{2}}.
\end{aligned}
\end{equation}
For the next term, we expand $B(u,\BB)=B(v,\BB)+B(w,\BB)+B(\BB,\BB)$ and estimate each separately. Starting with $B(v,\BB)$, we decompose it further to look at different horizontal and vertical parts separately: 
\begin{align*}
\langle B(v,\BB),v\rangle
=&\langle v^h\cdot\nabla^h\BB^h,v^h\rangle
+\langle v^h\cdot\nabla^h\BB^3,v^3\rangle\\
&+\langle v^3\partial_3\BB^h,v^h\rangle
+\langle v^3\partial_3\BB^3,v^3\rangle
\end{align*}
For the first term, we apply \eqref{aux-est2}, as $\| \mathcal{M}^h\|_{L^\infty_z} \le c$, followed by Young's inequality:
\begin{align*}
    \left|\langle v^h\cdot\nabla^h\BB^h,v^h\rangle \right| \le c \|w\|_{H^1} \| v \| \| \nabla^h v\| \le \delta \| \nabla^h v\|^2 + c_\delta \|v\|^2.
\end{align*}
We do the same for the next term, obtaining
\begin{align*}
    \left| \langle v^h\cdot\nabla^h\BB^3,v^3\rangle\right| &\le c(\varepsilon\nu)^{1/2}\|w\|_{H^2}\|v\| \|\nabla^hv\| \le \delta\|\nabla^hv\|^2+
 c_\delta \varepsilon \nu \|w\|_{H^2}^2\|v\|^2.
\end{align*}
For the third term, we start again with a horizontal H\"older
\begin{equation*}
    \left| \langle v^3\partial_3\BB^h,v^h\rangle \right| \le c \int_{[0,1]} \|v^3(\cdot,z)\|_{L^2_h} |\partial_3 \mathcal{M}^h| \|v^h(\cdot,z)\|_{L^4_h} dz.
\end{equation*}
Now, we inspect $\|v^3(\cdot,z)\|_{L^2_h}$ for $z\leq \frac{1}{2}$, the case $z\geq \frac{1}{2}$ is symmetric. Note that, due to the divergence free condition we have:
\begin{align*}
    \|v^3(\cdot,z)\|_{L^2_h} &= \norm{ \int_0^z \partial_3 v^3(\cdot, z') dz' }_{L^2_h} = \norm{ \int_0^z \divh v^h(\cdot, z') dz' }_{L^2_h} \\
    &\le \int_0^z \norm{\nabla^h v(\cdot, z')}_{L^2_h} dz'\le \sqrt{d(z)}\|\nabla^hv\|
\end{align*}
where in the second line we have used the Minkowski integral inequality, followed by Cauchy-Schwarz. Returning to the term at hand, we have applying Cauchy-Schwarz in $z$ twice:
\begin{align*}
    \left| \langle v^3\partial_3\BB^h,v^h\rangle \right| &\le c \|\nabla^hv\|\int_{[0,1]} \sqrt{d(z)} |\partial_3 \mathcal{M}^h| \|v^h(\cdot,z)\|_{L^4_h} dz \\
    &\le c \|\nabla^hv\| \|v\|_{L^2_zL^4_h} \left( \int_0^1 d(z) |\partial_3\mathcal{M}^h|^2 dz \right)^{\frac{1}{2}}\\
    &\le c \|\nabla^hv\| \|v\|_{L^2_zL^4_h} (\|\partial_3\mathcal{M}^h\|_{L^2_z}\|d\partial_3\mathcal{M}^h\|_{L^2_z} )^{\frac{1}{2}} \\
    &\le c \|\nabla^hv\| \|v\|_{L^2_zL^4_h} \le\  c \|\nabla^hv\|^{\frac{3}{2}} \|v\|^{\frac{1}{2}} \le \delta \| \nabla^h v\|^2 + c_\delta \|v\|^2,
\end{align*}
where the last estimate follows by Young's inequality with exponents $4, \frac{4}{3}$.
\noindent
For the last term, we have from \eqref{aux-est2}:
\begin{equation*}
    \left| \langle v^3\partial_3\BB^3,v^3\rangle \right| \le c \|v\|\|\nabla^h v\| \le \delta \| \nabla^h v\|^2 + c_\delta \|v\|^2,
\end{equation*}
as $\|\partial_3 \mathcal{M}^z\|_{L^\infty_z}\le c$ and $\|\curl w \|_{L^2_h}\le c$.
Collecting the above, we have obtained
\begin{equation}\label{eq:BvB}
    |\langle B(v,\BB),v\rangle| \le 4\delta\|\nabla^h v\|^2 + c_\delta\|v\|^2 + c_\delta \varepsilon\nu \|w\|_{H^2}^2\|v\|^2.
\end{equation}
For \(B(w,\BB)\), horizontal H\"older with exponents \(4,2,4\) followed by vertical Cauchy-Schwarz gives along with the corrector bounds \eqref{eq:matrix}:
\begin{equation*}
|\langle B(w,\BB),v\rangle|
\le c(\varepsilon\nu)^{1/4}\|w\|_{H^2}
 \|v\|^\frac{1}{2}\|\nabla^hv\|^\frac{1}{2}.
\end{equation*}
Two applications of Young's inequality give
\begin{equation}\label{eq:BwB}
    |\langle B(w,\BB),v\rangle| \le \delta \|\nabla^h v\|^2 + c_\delta(\varepsilon \nu)^{\frac{1}{2}}\|w\|_{H^2}^2 + c\|v\|^2.
\end{equation}
We finally move to the $B(\BB,\BB)$ term, which we decompose into $$\langle B(\BB,\BB), v\rangle = \langle \BB^h \cdot \nabla^h \BB, v\rangle + \langle \BB^3 \partial_3 \BB^h, v^h \rangle + \langle \BB^3 \partial_3 \BB^3, v^3 \rangle.$$
The first term of the three is bounded in exactly the same way as $\langle B(w,\BB),v\rangle$ bounding $\mathcal{M}^h$ uniformly, so
$$|\langle \BB^h \cdot \nabla^h \BB, v\rangle| \le \delta \|\nabla^h v\|^2 + c_\delta(\varepsilon \nu)^{\frac{1}{2}}\|w\|_{H^2}^2 + c\|v\|^2.$$
For the remaining two terms, we have:
$$
\begin{aligned}
|\langle\BB^3\partial_3\BB^h,v^h\rangle|
&\le c\|\mathcal M^z\|_{L_z^\infty}
\|\partial_3\mathcal M^h\|_{L_z^2}
\|\curl w\|_{L_h^2}\|w\|_{L_h^4}\|v\|_{L_z^2L_h^4}\\
&\le c(\varepsilon\nu)^{1/4}(\|v\|+\|\nabla^hv\|),\\
|\langle\BB^3\partial_3\BB^3,v^3\rangle|
&\le\|\mathcal M^z\|_{L_z^2}
\|\partial_3\mathcal M^z\|_{L_z^\infty}
\|\curl w\|_{L_h^4}^2\|v\|\\
&\le c(\varepsilon\nu)^{1/4}\|w\|_{H^2}\|v\|,
\end{aligned}
$$
where we have used the $2$-d Ladyzhenskaya inequality in the last line. Applying Young's inequality as before and collecting the above we have:
\begin{equation}\label{eq:BBB}
    |\langle B(\BB,\BB),v\rangle| \le \delta \|\nabla^h v\|^2 + c_\delta(\varepsilon \nu)^{\frac{1}{2}}(\|w\|_{H^2}^2 +1)+ c\|v\|^2.
\end{equation}
Next we treat the last of the nonlinear terms, along with the $\BB$ damping term. For this purpose, we will appeal to the crude bound $\|B(w,w)\|_{L^2_h} \le c \|w\|_{L^4_h} \|\nabla^hw\|_{L^4_h} \le c\|w\|_{H^2},$ and the estimate \eqref{eq:Bdual}:
\begin{align}\label{eq:Bnonlinear}
\left|\left\langle\BB[B(w,w)]+
\sqrt{2\nu/\varepsilon}\,\BB,v\right\rangle\right|
&\le c(\varepsilon\nu)^{1/4}(1+\|w\|_{H^2})
 (\|v\|+\|\nabla^hv\|) \nonumber\\
 & \le \delta \|\nabla^h v\|^2 + c_\delta(\varepsilon \nu)^{\frac{1}{2}}(\|w\|_{H^2}^2 +1)+ c\|v\|^2,
\end{align}
where the last line is just an application of Young's inequality.\\

\noindent We now deal with the remaining deterministic terms, in which we have to exploit the corrector cancellations \eqref{eq:profiles}.\\

\noindent By \eqref{eq:profiles}, the horizontal contributions of \(\BB_1\) and
\(\BB_2\) cancel exactly against rotation and
\(\sqrt{2\nu/\varepsilon}\,w\). The remaining horizontal terms are
\begin{align*}
    \left|\frac{2}{\varepsilon}\langle e_3 \wedge \BB_3^h, v\rangle \right| + \left|\frac{2}{\varepsilon}\langle e_3 \wedge \BB_4^h, v\rangle \right| &+ \left|\nu\langle A_z \BB_4^h, v\rangle \right| 
    \le c\left( \varepsilon^{-1}\| \BB_3[w]^h\| + \varepsilon^{-1}\| \BB_4[w]^h\| + \nu \| \partial^2_3 \BB_4[w]^h\| \right)\|v\|\\
    &\le c\left[
    \varepsilon^{-1}e^{-1/\sqrt{2\varepsilon\nu}}
    +\nu(\varepsilon\nu)^{-1/4}
    +\varepsilon^{-1}(\varepsilon\nu)^{3/4}
    +\nu \right]\|w\| \|v\|.
\end{align*}
We further note that due to $\frac{\nu}{\varepsilon}$ being bounded, we have $\varepsilon^{-1}e^{-1/\sqrt{2\varepsilon\nu}}
\le c(\varepsilon\nu)^{-1/2}
e^{-1/\sqrt{2\varepsilon\nu}}
\le c(\varepsilon\nu)^{1/4},$ as well as $\nu\le c(\varepsilon\nu)^{1/2}
\le c(\varepsilon\nu)^{1/4}$. Combining with the above, we have:
$$\left|\frac{2}{\varepsilon}\langle e_3 \wedge \BB_3^h, v\rangle \right| + \left|\frac{2}{\varepsilon}\langle e_3 \wedge \BB_4^h, v\rangle \right| + \left|\nu\langle A_z \BB_4^h, v\rangle \right| \le c (\varepsilon \nu)^{\frac{1}{4}}\|v\|.$$
For the vertical component, integration by parts,
\eqref{eq:B3derivative}, and \(\partial_3v^3=-\divh v^h\) give
\[
|\nu\langle\partial_3^2\BB^3,v^3\rangle|
\le\nu\|\partial_3\BB^3\|\|\partial_3v^3\|
\le c\nu(\varepsilon\nu)^{1/4}\|w\|_{H^1}\|\nabla^hv\|.
\]
We have thus proved
\begin{equation}\label{eq:Ekman-residual}
\begin{aligned}
&\left|\left\langle
\nu\PP\partial_3^2\BB
-\frac1\varepsilon\PP(e_3\wedge\BB)
+\sqrt{\frac{2\nu}{\varepsilon}}w,v
\right\rangle\right|\\
&\quad\le c(\varepsilon\nu)^{1/4}\|v\|
+c\nu(\varepsilon\nu)^{1/4}\|\nabla^hv\|\\
&\quad\le\delta\|\nabla^hv\|^2
+c\|v\|^2+c_{\delta}(\varepsilon\nu)^{1/2}.
\end{aligned}
\end{equation}
Collecting all the bounds on the deterministic terms (and making $\delta$ smaller by a factor of $100$), we see that they're bounded by an expression:
\begin{equation*}
\delta\|\nabla^hv\|^2+
 c_{\delta}(1+\|w\|_{H^2}^2)\|v\|^2+
 c_{\delta}(\varepsilon\nu)^{1/2}(1+\|w\|_{H^2}^2).
\end{equation*}
Thus, our energy inequality now takes the form:
\begin{align*}
    d\|v\|^2\le
    &-2 \|\nabla^hv\|^2 dt -2\nu \|\partial_3v\|^2 dt \\
    &+\sum_i \left( \langle\PP\Gt_iv,\Gt_i^*v\rangle dt+\langle\PP\Gt_i\BB,\Gt_i^*v\rangle dt-\langle\BB[\PP(\GG_i^h)^2w],v\rangle dt  \right) \\
    &+\sum_i\|\PP\Gt_iv+\PP\Gt_i\BB-\BB[\PP\GG_i^hw]\|^2 dt  \\
    &-2 \sum_i\langle\Gt_iv+\Gt_i\BB-\BB[\PP\GG_i^hw],v\rangle dW^i \nonumber\\
     &+(\delta\|\nabla^hv\|^2+
 c_{\delta}(1+\|w\|_{H^2}^2)\|v\|^2+
 c_{\delta}(\varepsilon\nu)^{1/2}(1+\|w\|_{H^2}^2)) dt\\
    &+ \sum_{i=1}^4 R_i \, dt + R_5 d\mathcal{W}_t.
\end{align*}
We now move to estimating the terms coming from the It\^o-Stratonovich corrector along with the quadratic variation terms.
Expanding the square and grouping terms, we have to estimate the following:
\begin{align*}
    &\sum_i\bigl(\langle\PP\Gt_iv,\Gt_i^*v\rangle+
    \|\PP\Gt_iv\|^2\bigr)\nonumber \\
    &-\sum_i\langle\BB[\PP(\GG_i^h)^2w],v\rangle\nonumber\\
    &+\sum_i\langle\PP\Gt_i\BB,\Gt_i^*v\rangle
    +2\sum_i\langle\PP\Gt_iv,
     \PP\Gt_i\BB-\BB[\PP\GG_i^hw]\rangle
    +\sum_i\|\PP\Gt_i\BB-\BB[\PP\GG_i^hw]\|^2.
\end{align*}
For the first term, we use \eqref{bigbound1NEW}:
$$\sum_i\bigl(\langle\PP\Gt_iv,\Gt_i^*v\rangle+
    \|\PP\Gt_iv\|^2\bigr) \le \delta(\|\nabla^hv\|^2+\nu\|\partial_3v\|^2)+c_\delta\|v\|^2,$$ as $\nu^{2\alpha}\le \nu$.

\noindent For the second term, we use the bound \eqref{eq:Bdual}, along with the standard estimates on $\Gt$, to get
$$\begin{aligned}
\sum_i|\langle\BB[\PP(\GG_i^h)^2w],v\rangle|
&\le \sum_i c(\varepsilon\nu)^{1/4}\|\PP(\GG_i^h)^2w\|_{L^2_h}
 (\|v\|+\|\nabla^hv\|)\\
&\le c(\varepsilon\nu)^{1/4}\|w\|_{H^2}
 (\|v\|+\|\nabla^hv\|)\\
&\le\delta\|\nabla^hv\|^2+c_\delta\|v\|^2
 +c_\delta(\varepsilon\nu)^{1/2}\|w\|_{H^2}^2.
\end{aligned}$$
Finally for the last line, we use the bounds \eqref{eq:Bnorm} along with Lemma \ref{bounds on split noise} to get:
\begin{align}
   \label{halfstuff} \|\PP\Gt_i\BB\| 
    &\le c\|\xi_i\|_{W^{1,\infty}}
    \bigl((\varepsilon\nu)^{1/4}\|w\|_{H^2}
    +\nu^\alpha(\varepsilon\nu)^{-1/4}\|w\|_{H^1}\bigr)\\ \nonumber
    &\le c\|\xi_i\|_{W^{1,\infty}}
    \bigl((\varepsilon\nu)^{1/4}\|w\|_{H^2}
    +\nu^{\alpha-\frac{1}{2}}\bigr)\\ \nonumber
    \|\BB[\PP\GG_i^hw]\| 
    &\le c\|\xi_i\|_{W^{1,\infty}}(\varepsilon\nu)^{1/4}\|w\|_{H^2}.
\end{align}
For the first two terms in the last line, we apply Cauchy-Schwarz to bound those terms by
$$c\bigl((\varepsilon\nu)^{1/4}\|w\|_{H^2}
+\nu^{\alpha-\frac{1}{2}}\bigr)
\bigl(\|\nabla^hv\|+\nu^\alpha\|\partial_3v\|+\|v\|\bigr) \le \delta(\|\nabla^hv\|^2 +\nu\|\partial_3v\|^2) +c_\delta(\|v\|^2
 +(\varepsilon\nu)^{1/2}\|w\|_{H^2}^2 + \nu^{2\alpha-1}).$$
Finally, the last term is handled similarly, and bounded by
\begin{equation*}
\begin{aligned}
\sum_i\bigl(\|\Gt_i\BB\|^2+\|\BB[\PP\GG_i^hw]\|^2\bigr)
&\le c\bigl((\varepsilon\nu)^{1/2}\|w\|_{H^2}^2
+\nu^{2\alpha-1}\bigr).
\end{aligned}
\end{equation*}
Collecting all the terms, we get that the It\^o-Stratonovich correction terms and quadratic variation terms are bounded by 
\begin{equation}\label{eq:noise-total}
\begin{aligned}
&\delta(\|\nabla^hv\|^2+\nu\|\partial_3v\|^2)+c_\delta\|v\|^2+c_\delta\bigl((\varepsilon\nu)^{1/2}\|w\|_{H^2}^2
+\nu^{2\alpha-1}\bigr)
\end{aligned}
\end{equation}
up to again changing $\delta$ by a factor. Updating our energy inequality, we now have
\begin{align}\label{eq:EE-last}
    d\|v\|^2\le
    &-2 \|\nabla^hv\|^2 dt\nonumber -2\nu \|\partial_3v\|^2 dt\\
    &-2 \sum_i\langle\Gt_iv+\Gt_i\BB-\BB[\PP\GG_i^hw],v\rangle dW^i \nonumber\\
     &+(\delta(\|\nabla^hv\|^2 + \nu\|\partial_3 v\|^2)+
 c_{\delta}(1+\|w\|_{H^2}^2)\|v\|^2+
 c_{\delta}(\varepsilon\nu)^{1/2}(1+\|w\|_{H^2}^2) + c_\delta \nu^{2\alpha -1})dt \nonumber\\
    &+ \sum_{i=1}^4 R_i \, dt + R_5 d\mathcal{W}_t.
\end{align}
We therefore move to bounding the stochastic integral using the BDG inequality. Anticipating the use of stochastic Gr\"onwall, we carry out the estimate between two arbitrary stopping times $0\le \theta \le \theta' \le T.$ Using the estimates appearing in the step above to control the quadratic variation term, along with \eqref{bigbound2NEW}, we have
\begin{equation*}
    \sum_i|\langle\Gt_iv+\Gt_i\BB-\BB[\PP\GG_i^hw],v\rangle|^2
    \le c\|v\|^4+
    c\bigl((\varepsilon\nu)^{1/2}\|w\|_{H^2}^2
    +\nu^{2\alpha - 1}\bigr)\|v\|^2.
\end{equation*}
The BDG inequality therefore gives
\begin{align*}
&\E\left[ \sup_{\theta\le r\le\theta'}\left|
2\sum_i\int_\theta^r
\langle\Gt_iv+\Gt_i\BB-\BB[\PP\GG_i^hw],v\rangle\,dW_s^{i}
\right| \right]\nonumber\\
&\quad\le c\E\left[\left(\sup_{\theta\le r\le\theta'}\|v_r\|^2\right)^{1/2}
 \left\{\int_\theta^{\theta'}\|v_s\|^2\,ds
 +c\bigl((\varepsilon\nu)^{1/2}
 +\nu^{2\alpha -1}\bigr)\right\}^{1/2}\right]
\nonumber\\
&\quad\le \delta \E\left[ \sup_{\theta\le r\le\theta'}\|v_r\|^2 \right]
 +c_\delta\E\left[ \int_\theta^{\theta'}\|v_s\|^2\,ds \right]
 +c_{\delta}\bigl((\varepsilon\nu)^{1/2}
 +\nu^{2\alpha-1}\bigr).
\end{align*}
Taking the supremum over $[\theta,\theta']$ and expectations in the inequality \eqref{eq:EE-last}, and taking $\delta\le \frac{1}{2}$ we have
\begin{align} \label{eq:EE-pre-G}
     \E\left[ \sup_{\theta\le r\le\theta'}\|v_r\|^2\right] &+ \E \left[ \int_\theta^{\theta'} (\|\nabla^hv_s\|^2 + \nu\|\partial_3 v_s\|^2) \, ds \right] \le \nonumber\\ 
     &\le c\E \left[ \|v_\theta\|^2 \right] + c\E \left[\sum_{i=1}^4 \int_0^T |R_i|\, ds \right]+ c\E \left[ \sup_{s \in [0,T]} \left| \int_0^sR_5 d \mathcal
     {W}_s\right|\right] + c_\delta((\varepsilon \nu)^{\frac{1}{2}} + \nu^{2\alpha-1})\nonumber\\
     &+ c_\delta \E\left[ \int_\theta^{\theta'}(1+\|w_s\|^2_{H^2})\|v_s\|^2\,ds\right],
\end{align}
where we have bounded $\E \int_\theta^{\theta'} \|w_s\|^2_{H^2} \, ds \le \E \int_0^T  \|w_s\|^2_{H^2} \, ds \le c.$
Let us now deal with the $R_i$ terms. We keep in mind, that we want those contributions to vanish as $n\to \infty$ while keeping $k,m$ fixed. 
\begin{align*}
    |R_1| \le c |\langle B(u,u),(I-\PP_n)(w+\BB[w])\rangle| &\le c \|u\|_{L^3}\|\nabla u\|
             \|(I-\PP_n)(w+\BB[w])\|_{L^6}\\
    &\le c\|u\|^{1/2}\|u\|_{H^1}^{3/2}
             \|(I-\PP_n)(w+\BB[w])\|_{H^1},
\end{align*}
where we used interpolation of \(L^3\) between \(L^2\) and \(L^6\), and a Sobolev embedding. We further note that
\begin{align*}
    \int_0^{T}\|u_n\|^{1/2}\|u_n\|_{H^1}^{3/2}\,ds
    &\le T^{1/4}(\sup_{s\le T}\|u_n(s)\|^2)^{1/4}
    \left(\int_0^T\|u_n\|_{H^1}^2\,ds\right)^{3/4}\\
    &\le c_T\left(\sup_{s\le T}\|u_n(s)\|^2+
    \int_0^T\|u_n\|_{H^1}^2\,ds\right),
\end{align*}
where we applied H\"older's inequality in the first line, and Young's in the second. Altogether, applying Cauchy-Schwarz and exploiting the additional probabilistic $L^4$ bound from \ref{item2.5}, we have 
\begin{equation}\label{eq:R1}
    \E \left[ \int_0^T |R_1|\,ds \right] \le c_k \left(\E\left[ \sup_{s\le T}
\|(I-\PP_n)(w_s+\BB[w_s])\|_{H^1}^2 \right]\right)^{1/2}.
\end{equation}
The terms $R_2$ and $R_3$ are bounded using simpler combination of the above tools, so we skip the details and simply give
\begin{equation}\label{eq:R2}
     \E \left[ \int_0^T |R_2|\,ds \right]  \le c \varepsilon^{-1}  \left(\E \left[ \sup_{s\le T}
\|(I-\PP_n)(w_s+\BB[w_s])\|_{L^2}^2\right]\right)^{1/2},
\end{equation}
as well as
\begin{equation}\label{eq:R3}
     \E \left[ \int_0^T |R_3|\,ds \right]  \le c \left(\E\left[ \sup_{s\le T}
\|(I-\PP_n)(w_s+\BB[w_s])\|_{H^1}^2 \right] \right)^{1/2}.
\end{equation}
We keep 
\begin{equation}
     \E \left[ \int_0^T |R_4|\,ds \right]= \E \left[ \int_0^T \sum_i \|(I-\PP_n)(\PP\GG_i^hw_s+\BB[\PP\GG_i^hw_s])\|^2 ds\right] .
\end{equation}
For the final term we apply the BDG inequality using the bound $$\sum_i|\langle\Gt_iu,(I-\PP_n)(w+\BB[w])\rangle|^2
\le c\|u\|^2\|(I-\PP_n)(w+\BB[w])\|_{H^1}^2.$$
We have:
\begin{align} \label{eq:R5}
    \E \left[ \sup_{s\in [0,T]} \left| \int_0^s R_5 d\mathcal{W}_s \right| \right] &\le c \E \left[ \left( \int_0^T\|u_s\|^2\|(I-\PP_n)(w_s+\BB[w_s])\|_{H^1}^2ds \right)^{\frac{1}{2}}\right] \nonumber\\
    & \le c \left(\E\left[ \sup_{s\le T}
    \|(I-\PP_n)(w_s+\BB[w_s])\|_{H^1}^2 \right]\right)^{1/2}.
\end{align}
Collecting \eqref{eq:R1}-\eqref{eq:R5}, we arrive at
\begin{align*}
    \E \left[ \sum_{i=1}^4 \int_0^T |R_i|\, ds \right]+ \E \left[ \sup_{s \in [0,T]} \left| \int_0^sR_5 d \mathcal
     {W}_s\right|\right] &\le c_k \left(\E \left[ \sup_{s\le T}
    \|(I-\PP_n)(w_s+\BB[w_s])\|_{H^1}^2\right]\right)^{1/2} +\\
    &+ \E \left[ \int_0^T \sum_i \|(I-\PP_n)(\PP\GG_i^hw_s+\BB[\PP\GG_i^hw_s])\|^2 ds\right].
\end{align*}
We will now argue that the right hand side goes to zero, as $n \to \infty$, for $k,m$ fixed. To this end, we note that $I-\PP_n$ converges strongly (in the operator sense) to $0$ over both $L^2_\sigma$ and $H^1$. We now note, that for a fixed function, $f\in C([0,T]; H^2)$, we have that the image of $[0,T]$ under $t\to f_t + \BB[f_t]$ is compact in $H^1$, and its image under $t \to  \PP\GG_i^hf_t+\BB[\PP\GG_i^hf_t]$ is compact in $ L^2_\sigma$ for each $i$. Thus, as pointwise convergence is uniform on compact sets, we have 
$$\sup_{t\le T}\|(I-\PP_n)(f_t+\BB[f_t])\|_{H^1}^2
\to 0, \text{ as $n \to \infty$},$$
and 
$$\sup_{t\le T} \sum_i \|(I-\PP_n)( \PP\GG_i^hf_t+\BB[\PP\GG_i^hf_t])\|_{L^2}^2
\to 0, \text{ as $n \to \infty$},$$
where taking the sum over $i$ is allowed, thanks to the assumptions on the $\xi_i$'s.
We further note, that using the corrector bounds \eqref{eq:Bnorm}
we have 
\begin{align*}
    \sup_{t\le T}\|(I-\PP_n)(f_t+\BB[f_t])\|_{H^1}^2 & \le c_{\varepsilon,\nu} \sup_{t\in[0,T]}\| f_t\|_{H^2}^2, \\
    \sup_{t\le T}\sum_i \|(I-\PP_n)( \PP\GG_i^hf_t+\BB[\PP\GG_i^hf_t])\|_{L^2}^2& \le c_{\varepsilon,\nu} \sup_{t\in[0,T]}\| f_t\|_{H^2}^2.
\end{align*}
We finally note, that the law of $w=w^{k,m,n}$ does not depend on $n$, and hence we can take the expectation with respect to that law (in particular by replacing $f$ by $w^{k,m}$), and pass the $n \to \infty$ limit through the expectation by the DCT. This application of the DCT is permitted, due to the additional regularity in the initial data, in particular for a fixed $k,m$, lemma \ref{uniform bounds on w} shows that $\E\left[ \sup_{t\in[0,T]}\| w^{k,m}_t\|_{H^2}^2\right] < \infty$. We summarise the above in writing: 
\begin{equation}
    \E\left[ \sum_{i=1}^4 \int_0^T |R_i|\, ds \right]+ \E \left[ \sup_{s \in [0,T]} \left| \int_0^sR_5 d \mathcal
     {W}_s\right| \right]\le o_n^{k,m},
\end{equation}
where $o_n^{k,m} \to 0$ as $n\to \infty$ for fixed $k,m$.

\noindent Coming back to \eqref{eq:EE-pre-G}, we have:
\begin{align} \label{eq:EE-Gready}
     \E\left[ \sup_{\theta\le r\le\theta'}\|v_r\|^2\right] + \E \left[ \int_\theta^{\theta'} (\|\nabla^hv_s\|^2 + \nu\|\partial_3 v_s\|^2) \, ds\right] &\le \nonumber c\E \left[ \|v_\theta\|^2\right] + o_n^{k,m} + c_\delta((\varepsilon \nu)^{\frac{1}{2}} + \nu^{2\alpha-1})\nonumber\\
     &+ c_\delta \E \left[ \int_\theta^{\theta'}(1+\|w_s\|^2_{H^2})\|v_s\|^2\,ds\right] .
\end{align}
Seeing that $$\int_0^T(1+\|w_s\|_{H^2}^2)\,ds\le c$$ with a deterministic constant, we apply the Stochastic Gr\"onwall lemma to obtain
\begin{equation}\label{eq:gronny-bound}
     \E\left[ \sup_{r\in [0,T]}\|v_r\|^2\right] + \E \left[ \int_0^{T} (\|\nabla^hv_s\|^2 + \nu\|\partial_3 v_s\|^2) \, ds \right]\le c_{\delta}\left( \E \left[ \|v_0\|^2\right] + o_n^{k,m} + c_\delta((\varepsilon \nu)^{\frac{1}{2}} + \nu^{2\alpha-1}) \right).
\end{equation}
Coming back to the beginning of the proof, denoting $v^{k,m,n} := u^{k,n} - w^{k,m,n} - \BB[w^{k,m,n}]$, we have
\begin{align*}
    \lim_{k 
       \rightarrow \infty} &\mathbbm{E}\left[\norm{u^k - w}_{L^\infty([0,T];L^2_{\sigma})}^2 \right] 
       \leq 4\lim_{k 
       \rightarrow \infty}\lim_{m \rightarrow \infty}\liminf_{n \rightarrow \infty}\mathbbm{E}\left[\norm{u^{k,n} - w^{k,m,n} }_{L^\infty([0,T];L^2_{\sigma})}^2\right] \\
       &\le c\lim_{k 
       \rightarrow \infty}\lim_{m \rightarrow \infty}\liminf_{n \rightarrow \infty}\mathbbm{E}\left[\norm{v^{k,m,n}}_{L^\infty([0,T];L^2_{\sigma})}^2 + \norm{\BB[w^{k,m,n}]}_{L^\infty([0,T];L^2_{\sigma})}^2\right] \\
       &\le c_\delta \lim_{k 
       \rightarrow \infty}\lim_{m \rightarrow \infty}\liminf_{n \rightarrow \infty} \left( \E \left[ \|u_0^{n}-w_0^{m} - \BB[w_0^{m}]\|^2 \right] + o_n^{k,m} +(\varepsilon_k \nu_k)^{\frac{1}{2}} + \nu_k^{2\alpha-1} \right) \\
       &+  \lim_{k 
       \rightarrow \infty}\lim_{m \rightarrow \infty}\liminf_{n \rightarrow \infty} (\varepsilon_k \nu_k)^{\frac{1}{2}}\E\left[  \|w^{k,m,n}\|_{L^\infty([0,T];H^1)}^2\right]\\
       &\le c_\delta \lim_{k 
       \rightarrow \infty}\lim_{m \rightarrow \infty}\left( \E \left[ \|u_0-w_0^{m} - \BB[w_0^{m}]\|^2 \right] +(\varepsilon_k \nu_k)^{\frac{1}{2}} + \nu_k^{2\alpha-1} \right) \\
       &\le  c_\delta \lim_{k 
       \rightarrow \infty}\lim_{m \rightarrow \infty}\left( \E \left[ \|u_0-w_0^{m}\|^2 \right] +(\varepsilon_k \nu_k)^{\frac{1}{2}} + \nu_k^{2\alpha-1} \right) \\
       &\le  c_\delta \lim_{k 
       \rightarrow \infty}\left((\varepsilon_k \nu_k)^{\frac{1}{2}} + \nu_k^{2\alpha-1} \right) =0,
\end{align*}
where we have used the energy estimate \eqref{eq:gronny-bound} in line three and the bound \eqref{eq:Bnorm} in line four and six.

\end{proof}

As a closing remark, we comment on the case $ \alpha = \frac{1}{2}$ and $\beta = 0$ which can be treated with the same method. The necessity of $\frac{1}{2} < \alpha$ appeared in the proof in the control of $\norm{\tilde{\mathcal{G}}_i\mathcal{B}}$, for example in (\ref{halfstuff}). The key point is that the quantity $\nu^{\alpha}(\varepsilon \nu)^{-\frac{1}{4}}$ must approach zero. We rewrite this as
$$\nu^{\alpha - \frac{1}{2}}\left(\frac{\nu}{\varepsilon}\right)^{\frac{1}{4}}.$$
In the proof, we use that $\frac{\nu}{\varepsilon}$ is bounded and that $\nu^{\alpha - \frac{1}{2}}$ approaches zero because $\frac{1}{2} < \alpha$. In the case $\beta = 0$, then $\alpha = \frac{1}{2}$ is permitted because $\frac{\nu}{\varepsilon}$ tends to zero. The rest of the proof is identical.\\

\textbf{Acknowledgements:} We would like to thank Franco Flandoli for generous and fruitful discussions on this problem.\\

\textbf{Competing Interests:} The authors report that there are no competing interests to declare.\\

\textbf{Data Availability Statement:} There is no associated data.

\addcontentsline{toc}{section}{References} 
\bibliographystyle{newthing}
\bibliography{Biblio}

\end{document}